\documentclass[12pt]{amsart}
\usepackage{latexsym, url}
\usepackage{amsthm}
\usepackage{amsmath}
\usepackage{amsfonts}
\usepackage{amssymb}
\usepackage[dvips]{graphicx}
\usepackage{xypic}
\usepackage{xcolor}
\usepackage{tikz}
\usetikzlibrary{arrows}
\usetikzlibrary{calc}
\usepackage{hyperref}
\input xy
\xyoption{all}

\newtheorem{theo}{Theorem}[section]
\newtheorem{lemma}[theo]{Lemma}
\newtheorem{assume}[theo]{Assumption}
\newtheorem{propo}[theo]{Proposition}
\newtheorem{coro}[theo]{Corollary}
\newtheorem*{theo:main}{Theorem~\ref{char-theo-1}}

\theoremstyle{definition}
\newtheorem{defi}[theo]{Definition}
\newtheorem{nota}[theo]{Notation}
\newtheorem{rem}[theo]{Remark}

\newtheorem{exam}[theo]{Example}
\newtheorem{exams}[theo]{Examples}

\newcommand\op{\operatorname{op}}
\newcommand\id{\operatorname{id}}

\newcommand\Set{\operatorname{\bf Set}}
\newcommand\Met{\operatorname{\bf Met}}

\newcommand\CMet{\operatorname{\bf CMet}}

\newcommand\Ban{\operatorname{\bf Ban}}

\newcommand\CPO{\operatorname{\bf CPO}}
\newcommand\MGra{\operatorname{\bf MGra}}

\newcommand\colim{\operatorname{colim}}

\newcommand\eps{\varepsilon}

\newcommand\ca{\mathcal {A}}

\newcommand\cc{\mathcal {C}}
\newcommand\cd{\mathcal {D}}

\newcommand\cu{\mathcal {U}}
\newcommand\cg{\mathcal {G}}
\newcommand\ch{\mathcal {H}}

\newcommand\ce{\mathcal {E}}

\newcommand\ck{\mathcal {K}}
\newcommand\cl{\mathcal {L}}
\newcommand\cm{\mathcal {M}}

\newcommand\cp{\mathcal {P}}
\newcommand\cq{\mathcal {Q}}
\newcommand\cw{\mathcal {W}}

\newcommand\cv{\mathcal {V}}

\newcommand{\tx}{\textnormal}
\newcommand{\bo}{\mathbf}
 
\date{December 9, 2024}
 
\begin{document}
\title[Notions of enriched purity]
{Notions of enriched purity}

\author[J. Rosick\'{y} and G. Tendas]
{J. Rosick\'{y} and G. Tendas}
\thanks{Both authors acknowledge the support of the Grant Agency of the Czech Republic under the grant 22-02964S. The second author also acknowledges the support of the EPSRC postdoctoral fellowship EP/X027139/1.} 
\address{
\newline J. Rosick\'{y}\newline
Department of Mathematics and Statistics\newline
Masaryk University, Faculty of Sciences\newline
Kotl\'{a}\v{r}sk\'{a} 2, 611 37 Brno, Czech Republic\newline
\textnormal{rosicky@math.muni.cz}\newline
\newline G. Tendas\newline
Department of Mathematics and Statistics\newline
Masaryk University, Faculty of Sciences\newline
Kotl\'{a}\v{r}sk\'{a} 2, 611 37 Brno, Czech Republic\vspace{5pt}\newline
\textit{Secondary address:}\newline
Department of Mathematics, University of Manchester,\newline 
Faculty of Science and Engineering, \newline
Alan Turing Building, M13 9PL Manchester, UK\newline
\textnormal{giacomo.tendas@manchester.ac.uk}
}
 
\begin{abstract}
We introduce enriched notions of purity depending on the left class $\ce$ of a factorization system on the base $\cv$ of enrichment. Ordinary purity is given by the class of surjective mappings in the category of sets. Under specific assumptions, covering enrichment over quantale-valued metric spaces, $\omega$-complete posets, and quasivarieties, we characterize the $(\lambda,\ce)$-injectivity classes of locally presentable 
$\cv$-categories in terms of closure under a class of limits, $\lambda$-filtered colimits, and $(\lambda,\ce)$-pure subobjects.
\end{abstract} 

\maketitle
\tableofcontents

\section{Introduction}
The concept of purity originated in the theory of abelian groups (\cite{P})
and it is now well established both for modules (\cite{Pre11}) and in model theory in general (\cite{Ro}). The monograph \cite{AR} made purity a key concept of the theory of accessible categories. For instance, given a locally $\lambda$-presentable category $\ck$, it makes possible to characterize classes of objects injective with respect to a family of morphisms whose domains and codomains are $\lambda$-presentable, as classes closed under products, $\lambda$-directed colimits, and $\lambda$-pure subobjects (see \cite{RAB}). Model-theoretically, this
is a characterization of classes of models axiomatizable by regular theories in the logic $L_{\lambda\lambda}$ of $\lambda$-ary formulas.

Recently, the theory of enriched accessible categories
matured (\cite{LT,LT22}) but purity has remained unenriched. 
The first attempt to develop enriched purity was done in \cite{RT} for metric enriched categories. In \cite{R}, it was shown that in Banach spaces pure morphisms coincide with the well-established concept of an ideal. Our aim is to introduce enriched purity over a general base $\cv$, prove that most of the basic properties transfer to the enriched setting, and show that it makes possible to characterize enriched injectivity classes. 

Recall that a morphism $f\colon K\to L$ in a $\lambda$-accessible category
is $\lambda$-pure provided that in each commutative diagram 
\begin{center}
		\begin{tikzpicture}[baseline=(current  bounding  box.south), scale=2]
			
			\node (a'0) at (0,0.8) {$A$};
			\node (b0) at (0.9,0.8) {$B$};
			\node (c0) at (0,0) {$K$};
			\node (d0) at (0.9,0) {$L$};
			
			\path[font=\scriptsize]

			(a'0) edge [->] node [above] {$g$} (b0)
			(b0) edge [->] node [right] {$v$} (d0)
			(c0) edge [->] node [below] {$f$} (d0)
			(a'0) edge [->] node [left] {$u$} (c0);
		\end{tikzpicture}	
	\end{center}
where $A$ and $B$ are $\lambda$-presentable, there is a morphism $t\colon B\to K$, such that $tg=u$ (see \cite[2.27]{AR}). Equivalently, given $u\colon A\to K$ then the existence of $v\colon B\to L$ such that $vg=fu$, implies the existence of $t\colon B\to K$ with $tg=u$. Since the {\em existence} of a morphism can be reformulated using surjectivity in the category $\Set$ of sets, ordinary purity stems from the class $\ce$ of surjections in $\Set$. 

Building on \cite{LR12}, our enriched purity will depend on a chosen class $\ce$ of morphisms in $\cv$ which, in addition, arises as the left class of a factorization system. In metric-enriched categories these are the dense morphisms. 

In Sections~\ref{enr-purity1} and~\ref{enr-purity2} we propose two definition of $(\lambda,\ce)$-pure morphisms. The first one is stronger than the second but, under mild conditions, they coincide. These behave like $\lambda$-pure morphisms in ordinary categories; for instance we will see that they are left-cancellable, stable under composition, and satisfy many of the same properties as their unenriched counterpart. Moreover, they make possible the characterization of $(\lambda,\ce)$-injectivity classes (this enriched concept of injectivity was introduced in \cite{LR12}).

In metric-enriched categories, we get the notions of purity from \cite{RT} and the characterization of approximate $\lambda$-injectivity classes (that is, ($\lambda$, dense)-injectivity classes) proved in \cite{RT}. In additive categories, when considering the class of regular epimorphisms, enriched and unenriched purity coincide, as do enriched and unenriched injectivity classes: this is the case whenever we enrich over a symmetric monoidal quasivariety~(\cite{LT20}) with a regular projective unit. However, things change when considering, for instance, categories enriched over the symmetric monoidal category $\bo{DGAb}$ of differentially graded abelian groups, and still considering the class $\ce$ of the regular epimorphisms. Here, it became clear that ordinary purity could not be used to characterize such injectivity classes; the main obstacle being that, unlike in the previous situation, the unit of $\bo{DGAb}$ is not regular projective. We shall see that the corresponding DG-enriched notion of purity does not coincide with the ordinary one, and is in fact the missing piece needed in the characterization of DG-injectivity.

Our main result characterizing $(\lambda,\ce)$-injectivity classes is discussed in Section~\ref{E-is-inj} and is based on the assumption that the class $\ce$ itself is a (special) injectivity class in the ordinary category of arrows $\cv_0^\to$. This allows us to capture surjections (Remark~\ref{surjection1}), dense morphisms of (quantale-valued) metric spaces and Banach spaces, dense maps of $\omega$-complete posets, as well as regular epimorphisms in quasivarieties, all in the same general framework. Beside this, and some minor other assumptions, we shall also need a class of objects $\cg\subseteq\cv$ such that powers by $\cg$ satisfy a stability condition with respect to maps in $\ce$. Then we can prove:

\begin{theo:main}
	Under the assumptions above, the $(\lambda,\ce)$-injectivity classes in a locally $\lambda$-presentable $\cv$-category $\ck$ are precisely the classes closed under $\lambda$-filtered colimits, products, powers by $\cg$, and $(\lambda,\ce)$-pure subobjects.
\end{theo:main}

For $\cv=\Met,\CMet,\Ban$ or $\omega$-$\CPO$ with $\ce$ consisting of the dense maps, then $\cg=\{I\}$ is the singleton determined by the monoidal unit (so that powers by $\cg$ are trivial); instead, when $\cv=\bo{DGAb}$ and $\ce$ is formed by the regular epimorphisms, then $\cg=\{P_n\}_{n\in \mathbb Z}$ is given by the suspensions and de-suspensions of the mapping cone of the unit (Example~\ref{DG}). 

As a consequence of the theorem above, we prove (Corollary~\ref{orth->inj}) that enriched orthogonality classes in locally presentable $\cv$-categories arise as examples of $\ce$-injectivity classes, a well-known fact in ordinary category theory \cite[Chapter~4.A]{AR}.

We will discuss applications of Theorem~\ref{char-theo-1} in the setting of categories enriched over quantale-valued metric spaces in Section~\ref{quantales}, for categories enriched over $\omega$-complete poses in Section~\ref{omega-CPO}, and for categories enriched over symmetric monoidal quasivarieties in Section~\ref{quasiv}. In this latter case, we can also characterize $(\lambda,\ce)$-pure morphisms as $\lambda$-filtered colimits of {\em $\ce$-split} monomorphisms, generalizing the analogous result for ordinary categories \cite[2.30]{AR}.

While in this paper we already answer important questions about the relevance of enriched purity, there are still many open problems that wait to be addressed. For instance, when working with categories of structures of some relational language, ordinary pure morphisms correspond to the model theoretical notion of morphisms {\em elementary} with respect to positive-primitive formulas \cite[5.15]{AR}. A first glance at an enriched counterpart is given in Appendix~\ref{can-lang}. However, a general answer cannot be provided at this stage for the simple reason that enriched counterparts of these logical concepts have not yet been defined. We believe that a thorough study of enriched purity will be essential for the development of enriched notions of languages and models, thus covering an outstanding gap in the literature.

\section{Background}

Our base of enrichment will be a symmetric monoidal closed, complete and cocomplete category $\cv=(\cv_0,I,\otimes)$; the internal-hom is denoted by $[-,-]$ and makes $\cv$ into a $\cv$-category. We follow \cite{Kel82:book} for standard results about enrichment; in particular we shall make use of the notions of conical limits and colimits, and of powers and copowers, both arising as instances of the more general weighted limits and colimits.

Consider a $\cv$-category $\ck$ and an ordinary functor $H\colon\cc\to\ck_0$ into the underlying ordinary category of $\ck$; this corresponds to a $\cv$-functor $\cc_\cv\to\ck$ from the free $\cv$-category on $\cc$. When $\cc$ is small, the {\em (conical) limit of $H$ in $\ck$} is the data of an object $\lim H\in\ck$ together with isomorphisms in $\cv$ 
$$ \ck(K,\lim H)\cong \lim\ck(K,H-) $$
that are $\cv$-natural in $K\in\ck$, where on the right-hand-side we have the ordinary limit in $\cv_0$. It follows that $\lim H$, when it exists, is also the limit of $H$ in $\ck_0$; moreover $\lim H$ corresponds to the limit of $H$ weighted by $\Delta I\colon\cc_\cv\to\cv$ (see \cite[3.51]{Kel82:book}).

Given an object $X\in\cv$ and $K\in\ck$, the {\em power of $K$ by $X$} is the data of an object $X\pitchfork K\in\ck$ together isomorphisms in $\cv$
$$ \ck(L,X\pitchfork K)\cong [X,\ck(L,K)] $$
that are $\cv$-natural in $L\in\ck$. Dually, we denote by $X\cdot K$ the {\em copower of $K$ by $X$}. Note that when $\ck=\cv$, we have $X\pitchfork K\cong [X,K]$ and $X\cdot K\cong X\otimes K$.

Local presentability will also be a key concept in the following sections. For ordinary locally presentable categories our standard reference is \cite{AR}. In the enriched setting the concept was introduced by Kelly \cite{Kel82}: given a $\cv$-category $\ck$ with (conical) $\lambda$-filtered colimits, we say that $K\in\ck$ is {\em $\lambda$-presentable} (in the enriched sense) if $\ck(K,-)\colon\ck\to\cv$ preserves $\lambda$-filtered colimits. We denote by $\ck_\lambda$ the full subcategory of $\ck$ spanned by the $\lambda$-presentable objects; this is closed in $\ck$ under $\lambda$-small conical colimits and copowers by objects in $\cv_\lambda$. 

We say that $\ck$ is {\em locally $\lambda$-presentable as a $\cv$-category} if it is cocomplete (it has all conical colimits and all copowers), $\ck_\lambda$ is (essentially) small, and every object of $\ck$ can be written as a $\lambda$-filtered colimit of $\lambda$-presentable objects. If the unit $I$ of $\cv_0$ is locally $\lambda$-presentable in the ordinary sense, then the underlying category of every locally $\lambda$-presentable $\cv$-category is locally $\lambda$-presentable and $(\ck_\lambda)_0=(\ck_0)_\lambda$.

For the remainder of this paper we shall assume our base of enrichment $\cv$ to be {\em locally $\lambda$-presentable as a closed category} \cite{Kel82}, meaning that $\cv_0$ is locally $\lambda$-presentable, $I$ is $\lambda$-presentable and the $\lambda$-presentable are closed under tensor product (in particular $\cv$ is then locally $\lambda$-presentable as a $\cv$-category). In that context, many results from the ordinary theory of local presentability extend to the enriched context; see \cite{Kel82} for details. 

Finally, (orthogonal) factorization systems will be central in the definition of our enriched notions of purity. Following \cite{FK} we will say that a factorization system $(\ce,\cm)$ on a category $\ck$ is {\em proper} if every element of $\ce$ is an epimorphism and every element of $\cm$ a monomorphism.  By \cite[2.1.4]{FK}, when the factorization system is proper its is true that whenever a composite $fg\in\ce$ then also $f\in\ce$; we will make use of this property quite often.

In Section~\ref{E-is-inj} we will assume our fixed factorization system $(\ce,\cm)$ on $\cv$ to be {\em enriched} in the sense of \cite{Luc14}. This means that the class $\ce$ is closed in $\cv^\to$ under all copowers (if $e\in\ce$ and $X\in\cv$, then $X\otimes e\in\ce$), or equivalently that $\cm$ is closed in $\cv^\to$ under all powers (if $m\in\cm$ and $X\in\cv$, then $[X,m]\in\cm$).

\section{Enriched purity}\label{enr-purity1}

We fix $\cv$ to be locally $\lambda$-presentable as a closed category and $(\ce,\cm)$ to be factorization system on $\cv_0$. In this section we introduce the first of two concepts of enriched purity and prove several properties that this satisfies.

\begin{nota}\label{strong-pure}
	Let $\ck$ be a $\cv$-category and $(\ce,\cm)$ a factorization system on $\cv$. Consider morphisms $f\colon K\to L$ and $g\colon A\to B$ in $\ck$. We denote by $\cp(g,L)$ be the $(\ce,\cm)$ factorization 
		\begin{center}
		\begin{tikzpicture}[baseline=(current  bounding  box.south), scale=2]
			
			\node (21) at (0,0) {$\ck(B,L)$};
			\node (22) at (1.3,0) {$\cp(g,L)$};
			\node (23) at (2.6,0) {$\ck(A,L)$};
			
			\path[font=\scriptsize]
			
			(21) edge [->>] node [above] {} (22)
			(22) edge [>->] node [above] {} (23);
		\end{tikzpicture}	
	\end{center} 
	of the map $\ck(g,L)$ in $\cv$. Let then $\cp(g,f)$ be the pullback of $\cp(g,L)$ along $\ck(A,f)$, together with the induced map $r\colon\ck(B,K)\to\cp(g,f) $ as depicted below.
	\begin{center}
		\begin{tikzpicture}[baseline=(current  bounding  box.south), scale=2]
			
			\node (a0) at (-0.5,1.2) {$\ck(B,K)$};
			\node (b0) at (1,1.2) {$\ck(B,L)$};
			\node (c0) at (0,-0.2) {$\ck(A,K)$};
			\node (d0) at (1.5,-0.2) {$\ck(A,L)$};
			\node (b'0) at (0,0.55) {$\cp(g,f)$};
			\node (c'0) at (1.5,0.55) {$\cp(g,L)$};
			
			\path[font=\scriptsize]

			(a0) edge [->] node [above] {$\ck(B,f)$} (b0)
			(b'0) edge [->] node [above] {$r'$} (c'0)
			(b0) edge [->>] node [right] {} (c'0)
			(c'0) edge [>->] node [right] {} (d0)
			(c0) edge [->] node [below] {$\ck(A,f)$} (d0)
			(a0) edge [bend right,->] node [left] {$\ck(g,K)$} (c0)
			(a0) edge [dashed, ->] node [right] {$r$} (b'0)
			(b'0) edge [>->] node [right] {} (c0);
		\end{tikzpicture}	
	\end{center}
\end{nota}

\begin{defi}\label{pure-def1}
	We say that $f\colon K\to L$ is $\ce$-\textit{pure with respect to $g$} if the map $r$ above lies in $\ce$. We say that $f$ is $(\lambda,\ce)$-\textit{pure} if it is $\ce$-pure with respect to every $g\colon A\to B$ with $A$ and $B$ $\lambda$-presentable.
\end{defi}

Remark~\ref{explanation-purity} below gives a possibly more intuitive interpretation of the notion just introduced. In (1) of Example \ref{ex1} below we show that, over $\Set$, we get the standard notion of purity. The reader can also look at a model theoretic characterization of purity from Appendix~\ref{can-lang} in terms of {\em elementary morphisms}. We will spell-out in Sections~\ref{quantales}, \ref{omega-CPO}, and \ref{quasiv} what purity looks like for the bases of enrichment we are most interested in.

\begin{rem}\label{explanation-purity} 
	Since the elements of $\cm$ are stable under pullback, the map $r$ above lies in $\ce$ if and only if $\cp(g,f)$ is the $(\ce,\cm)$-factorization of $\ck(g,K)$. Thus $f\colon K\to L$ is $(\lambda,\ce)$-pure if and only if:
	\begin{center}{\em
		for every $g\colon A\to B$, with $A$ and $B$ $\lambda$-presentable,\\ 
		the $(\ce,\cm)$-factorization of $\ck(g,K)$ is obtained by pulling back the\\ $(\ce,\cm)$-factorization of $\ck(g,L)$ along $\ck(A,f)$.}
	\end{center} 
\end{rem}

\begin{rem}\label{surjection1}
	Following \cite{R1}, a morphism $f\colon V\to W$ in $\cv$ is called 
	\textit{surjective} if $\cv_0(I,f)$ is a surjection of sets. A map is called {\em injective} if it satisfies the unique right lifting property with respect to surjections. Quite often, (surjective, injective) forms a proper factorization system on $\cv$ (see \cite{R1}). 
\end{rem}

\begin{exams}\label{ex1}$  $
	{\setlength{\leftmargini}{1.6em}
	\begin{enumerate}
	\item Consider $\cv$ such that (surjective, injective) is a factorization system, and note that by \cite[3.4]{R1}, injective maps are monomorphisms. Under Notation~\ref{strong-pure}, it follows that $\cv_0(I,\cp(g,L))$ is the (epi, mono) factorization of $\ck_0(g,L)$, so that
	$$\ \ \ \ \cv_0(I,\cp(g,L))=\{ h\in\ck_0(A,L)\ |\ \exists\ v\in\ck_0(B,L) \text{ with } h= vg\} $$
	and hence, since $\cv_0(I,-)$ preserves pullbacks,
	$$ \ \ \ \ \cv_0(I,\cp(g,f))=\{ u\in\ck_0(A,K)\ |\ \exists\ v\in\ck_0(B,L) \text{ with } fu= vg\}. $$
	Then, $f$ is $\ce$-pure with respect to $g$, if and only if $\cv_0(I,r)$ is surjective, if and only if:
	\begin{center}\em
		for any $u\colon A\to K$ and $v\colon B\to L$ with $fu=vg$,\\ there exists $t\colon B\to K$ such that $tg=u$.
	\end{center}
	As a consequence, $\ce$-purity in $\ck$ coincides with ordinary purity in $\ck_0$. In particular, for $\cv=\Set$ we obtain the usual notion of pure morphisms.
	
	\item Given $\cv$, consider the factorization system $(\ce,\cm)=(\cv^{\to},\cv^{\cong})$, where $\cv^{\cong}$ is the class of all isomorphisms and $\cv^{\to}$ the class of all morphisms. Then, every morphism $f$ is $\cv^{\to}$-pure with respect to any $g\colon A\to B$; indeed, the map $r\colon\ck(B,K)\to\cp(g,f)$ of Notation~\ref{strong-pure} always lies in $\cv^{\to}$. 
	
	\item Consider instead the factorization system $(\ce,\cm)=(\cv^{\cong},\cv^{\to})$ on $\cv$, so that $\ce$ is the class of all isomorphisms. Given $f\colon K\to L$ and $g\colon A\to B$, it follows that $\cp(g,L)\cong\ck(B,L)$ and therefore $$\cp(g,f)\cong\ck^\to(g,f)$$ is the hom-object in the $\cv$-category of arrows $\ck^\to$. Then $f$ is $\cv^{\cong}$-pure with respect to $g$ if and only if $$\ck(B,K)\cong \ck^\to(g,f);$$ 
	looking at the underlying categories, this implies in particular that for every pair $u\colon A\to K$ and $v\colon B\to L$, such that $fu=vg$, there is a {\em unique} $t\colon B\to K$ such that $ft=v$ and $tg=u$.
	
	For $\ck$ locally $\lambda$-presentable, $(\lambda,\cv^{\cong})$-pure morphisms then coincide with the isomorphisms. On the one hand, isomorphisms are always $\cv^{\cong}$-pure with respect to any morphism. Conversely, assume that $f$ is $(\lambda,\cv^{\cong})$-pure, and write it as a $\lambda$-directed colimit of morphisms $f_i\colon K_i\to L_i$ between $\lambda$-presentable objects (see \cite[1.55(1)]{AR}); then for any $i$ we can consider the square 
	\begin{center}
		\begin{tikzpicture}[baseline=(current  bounding  box.south), scale=2]
			
			\node (a'0) at (0,0.8) {$K_i$};
			\node (b0) at (0.9,0.8) {$L_i$};
			\node (c0) at (0,0) {$K$};
			\node (d0) at (0.9,0) {$L$};
			
			\path[font=\scriptsize]

			(a'0) edge [->] node [above] {$f_i$} (b0)
			(b0) edge [->] node [right] {$v_i$} (d0)
			(c0) edge [->] node [below] {$f$} (d0)
			(a'0) edge [->] node [left] {$u_i$} (c0);
		\end{tikzpicture}	
	\end{center}
	where $(u_i,v_i)$ is part of the colimiting cocone of $f$. By $\cv^{\cong}$-purity of $f$, for any $i$ there exists a unique $t_i\colon L_i\to K$ such that $ft_i=v_i$ and $t_if_i=u_i$. It is easy to see that the $(t_i,v_i)$ also define a colimiting cocone for $f$ in $\cv^\to$. Therefore $f$ is a colimit of the isomorphisms $1_{L_i}\colon L_i\to L_i$, and thus $f$ itself is an isomorphism. 
	
	\item  Let $\CMet$ be the category of complete metric spaces and nonexpanding maps, where we allow distances $\infty$. Here, we have the proper factorization system (dense, isometry) --- see \cite[3.16(2)]{AR1}.
	
	Under this factorization system, given $g$ and $f$ as above, $\cp(f,g)$ is the sub-metric space of $\ck(A,K)$ consisting of those $u\colon A\to K$ such that
	\begin{center}\em 
		for every $\eps>0$ there exists $v\colon B\to L$ for which\\ $fu\sim_{\eps}vg$ (that is,  $d(fu,vg)\leq\eps$)
	\end{center} 
	Then $f$ is $\ce$-pure with respect to $g$ if and only if for any $u\in \cp(f,g)$ as above and for any $\eps>0$, there exists $t\colon B\to K$ such that $tg\sim_{\eps}u$.
	
	A priori this doesn't correspond to a notion of purity studied for metric-enriched categories. However, we will see in the more general setting of Section~\ref{quantales}, that this is an equivalent reformulation of the notions considered in \cite{RT}.
	
	\item  In the category $\Met$ of metric spaces and nonexpanding maps, we have two factorization systems: (dense, closed isometry) and (surjective, isometry) --- see \cite[3.16]{AR1}. The first one yields the same concept of purity described in (4) above, while the second one gives the usual (unenriched) concept of purity, as in (1). 
	
	Observe that surjections in $\CMet$ are the usual surjective maps but
	(surjective, injective), as defined in Remark~\ref{surjection1}, does not form a factorization system. For instance the morphism
	$\Bbb N\to \{\frac{1}{n}|n\in\Bbb N\}\cup\{0\}$ (sending $n$ to $\frac{1}{n}$), where the domain is discrete and the codomain has the usual metric, does not factorize.
	
	\item  Let $\cv=\bo{Gra}$ be the cartesian closed category of directed graphs (possibly) without loops. Consider the (regular epi, mono) factorization system. Given $f\colon K\to L$, this is pure in the enriched sense if for any morphism $g\colon A\to B$ between finitely presentable objects, and any vertices $u\in\ck(A,K)$ and $v\in\ck(B,L)$ for which $vg=fu$, there exists a vertex $t\in\ck(B,K)$ such that $tg=u$. Since not every vertex is a morphism in the underlying category (only the vertices with a loop are such), the notion of enriched purity is not the ordinary one on the underlying category; this is essentially because the unit is not regular projective. 
	
	\item  Let $\cv=\bo{DGAb}$ be the category of chain complexes with the (regular epi, mono) factorization. The notion of $\ce$-purity does not coincide with the ordinary one, as we see in more detail in Section~\ref{quasiv}.
	
	\item   Let $\cv=\bo{Ban}$ be the category of Banach spaces over $\mathbb C$ with linear maps of norm at most $1$. This is symmetric monoidal closed with $\otimes$ given by the projective tensor product, and internal hom $[K,L]$ is the space consisting of all bounded linear operators from $K$ to $L$ (see \cite[6.1.9h]{B}). Note that the monoidal unit $\mathbb C$ is a strong generator: $\bo{Ban}(\mathbb C,-)$ is the unit-ball functor, and an operator is an isomorphism in $\bo{Ban}$ if and only if it is a bijection on the unit balls. 
	
	Here, we can consider the (epi, strong mono) factorization system, which by \cite[1.15~\&~2.3]{Po} coincides with the (dense, isometry) factorization system; this induces a notion of purity similar to that of $\bo{Met}$ --- see in particular Examples~\ref{Ban1} and~\ref{Ban2}. 
	
	We can also consider the (strong epi, mono) factorization system. It follows from \cite[2.7~\&~2.10]{Po}  that regular and strong epimorphisms in $\bo{Ban}$ coincide.  Moreover, they are just the surjective operators, or equivalently, the operators $T$ for which $\bo{Ban}_0(\mathbb C, T)$ is surjective (\cite{Po} shows that strong/regular epimorphisms are surjective, but since homming-out of a strong generator reflects strong epimorphisms, then also every surjective operator is a strong epimorphism). Since $I=\Bbb C$, these coincide with the surjections from Remark~\ref{surjection1}. Thus the notion of purity corresponding to this factorization system coincides with the ordinary one on the underlying category.
\end{enumerate}}
\end{exams}

\begin{rem}
	If $\lambda\leq\lambda'$ then $(\lambda',\ce)$-pure morphisms are
	$(\lambda,\ce)$-pure.
\end{rem}

In the ordinary setting, $\lambda$-pure morphisms are closed under composition and satisfy the left-cancellation property (see \cite{AR} after 2.28). The same holds for $(\lambda,\ce)$-pure morphisms, as we now see:

\begin{propo}\label{composition1}
	A composition of two $(\lambda,\ce)$-pure morphisms is $(\lambda,\ce)$-pure.
\end{propo}
\begin{proof}
	Consider two $(\lambda,\ce)$-pure morphisms $f_1\colon K\to L$ and $f_2\colon L\to M$,
	and a morphism $g\colon A\to B$ between $\lambda$-presentable objects. Let $f=f_2f_1$. In the diagram below
	\begin{center}
		\begin{tikzpicture}[baseline=(current  bounding  box.south), scale=2]
			
			\node (21) at (0,0) {$\ck(B,K)$};
			\node (22) at (1.5,0) {$\ck(B,L)$};
			
			\node (31) at (0,-0.8) {$\cp(g,f)$};
			\node (32) at (1.5,-0.8) {$\cp(g,f_2)$};
			\node (33) at (3,-0.8) {$\cp(g,M)$};
			
			\node (41) at (0,-1.6) {$\ck(A,K)$};
			\node (42) at (1.5,-1.6) {$\ck(A,L)$};
			
			\node (23) at (3,0) {$\ck(B,M)$};
			\node (43) at (3,-1.6) {$\ck(A,M)$};

			\node (c) at (0.75,-1.2) {$(a)$};
			\node (d) at (2.25,-1.2) {$(b)$};
			
			\path[font=\scriptsize]
			
			(21) edge [->] node [below] {} (22)
			(21) edge [->] node [left] {$r$} (31)
			(22) edge [->>] node [right] {$r_2$} (32)
			(31) edge [->] node [above] {} (32)
			(31) edge [>->] node [left] {} (41)
			(32) edge [>->] node [right] {} (42)
			(41) edge [->] node [below] {$\ck(A,f_1)$} (42)
			(22) edge [->] node [above] {} (23)
			(23) edge [->>] node [right] {} (33)
			(33) edge [>->] node [right] {} (43)
			(32) edge [->] node [above] {} (33)
			(42) edge [->] node [below] {$\ck(A,f_2)$} (43);
		\end{tikzpicture}	
	\end{center} 
	the squares $(a)$, $(b)$, and $(a+b)$ are pullbacks. By $\ce$-purity of $f_2$, the map $r_2$ lies in $\ce$ and thus $\cp(g,f_2)$ coincides with the $(\ce,\cm)$-factorization of $\ck(g,L)$. In other words $\cp(g,f_2)\cong \cp(g,L)$. Then, since $(a)$ is a pullback, $\cp(g,f)$ is also the pullback of $\cp(g,L)\rightarrowtail\ck(A,L)$ along $\ck(A,f_1)$; thus $\cp(g,f)\cong \cp(g,f_1) $ and $r\cong r_1$. Since $f_1$ is $(\lambda,\ce)$-pure this means that $r$ is in $\ce$, as required.
\end{proof}

\begin{propo}\label{cancell1}
	If $(\ce,\cm)$ is proper, then $(\lambda,\ce)$-pure morphisms are left-cancellable; that is, if $f=f_2f_1$ is $(\lambda,\ce)$-pure then $f_1$ is $(\lambda,\ce)$-pure.
\end{propo}
\begin{proof}
	Let $f=f_2f_1$ be $(\lambda,\ce)$-pure. Then there is a morphism $s\colon\cp(g,f_1)\to\cp(g,f)$ induced by the universal property of the pullback defining $\cp(g,f)$; this $s$ makes the diagram below commute,
	\begin{center}
		\begin{tikzpicture}[baseline=(current  bounding  box.south), scale=2]
			
			\node (a0) at (0,0.8) {$\ck(B,K)$};
			\node (b0) at (1.3,0.8) {$\cp(g,f_1)$};
			\node (c0) at (0,0) {$\cp(g,f)$};
			\node (d0) at (1.3,0) {$\ck(A,K)$};
			
			\path[font=\scriptsize]

			(a0) edge [->] node [above] {$r_1$} (b0)
			(b0) edge [>->] node [right] {} (d0)
			(c0) edge [>->] node [below] {} (d0)
			(b0) edge [->] node [below] {$s$} (c0)
			(a0) edge [->>] node [left] {$r$} (c0);
		\end{tikzpicture}	
	\end{center}
	where $r_1$ denotes the dashed map from Notation~\ref{strong-pure} corresponding to $f_1$.
	Since the morphisms into $\ck(A,K)$ are in $\cm$, so is $s$. Moreover, since the factorization system is proper and $r$ is in $\ce$, also $s$ is in $\ce$. Therefore $r$ is an isomorphism and $r_1$ is in $\ce$; thus $f_1$ is $(\lambda,\ce)$-pure.
\end{proof}

Another property of $\lambda$-pure morphism that extends to our enriched notion, is their closure under $\lambda$-filtered colimits. We do not know whether a result like \cite[Proposition~2.30(ii)]{AR}, expressing every $\lambda$-pure morphism in a locally $\lambda$-presentable category as a $\lambda$-filtered colimit of split monomorphisms, can be obtained in this generality. However, something can be achieved when we enrich over quasivarieties, see Proposition~\ref{filt colimit of E-split}. 

\begin{propo}\label{filt-col-pure1}
	If $\cm$ is closed under $\lambda$-filtered colimits in $\cv^\to$, then any $\lambda$-filtered colimit of $(\lambda,\ce)$-pure morphisms in $\ck^\to$ is $(\lambda,\ce)$-pure.
\end{propo}
\begin{proof}
	Let $f\colon K\to L$ be a $\lambda$-filtered colimit in $\ck^\to$ of some $(\lambda,\ce)$-pure morphisms $f_i\colon K_i\to L_i$, and consider any $g\colon A\to B$ in $\ck_\lambda$. Since $\ck(A,K)$ is the $\lambda$-filtered colimit of $\ck(A,K_i)$ and $\ck(B,K)$ is the $\lambda$-filtered colimit of $\ck(B,K_i)$, the map 
$$
\ck(g,K)\colon \ck(B,K)\to \ck(A,K)
$$ 
is the $\lambda$-filtered colimit of the maps 
$$
\ck(g,K_i)\colon \ck(B,K_i)\to \ck(A,K_i);
$$ 
the same applies for $L$ in place of $K$. Now, by $(\lambda,\ce)$-purity of each $f_i$, the $(\ce,\cm)$-factorization of each $\ck(g,K_i)$ is given by the object $\cp(g,f_i)$, see Remark~\ref{explanation-purity}. By orthogonality of the factorization system, we obtain a $\lambda$-filtered diagram on the $\cp(g,f_i)$ whose colimit $X(g,f)\in\cv$ induces the $(\ce,\cm)$-factorization of $\ck(g,K)$; this is because both $\ce$ and $\cm$ are closed under $\lambda$-filtered colimits in $\ck^\to$. But $\lambda$-filtered colimits commute with pullbacks in $\cv$; thus $X(g,f)$ is also the pullback along $\ck(A,f)$ of the colimit of $\cp(g,L_i)$ . Since this colimit coincides with $\cp(g,L)$ then $X(g,f)\cong\cp(g,f)$ and thus $f$ is $(\lambda,\ce)$-pure with respect to $g$.
\end{proof}

As a consequence we can show that under some mild assumptions, in a locally $\lambda$-presentable $\cv$-category the ordinary notion of purity is always stronger than the enriched one:

\begin{coro}\label{pure-Epure}
	If $(\ce,\cm)$ is proper and $\cm$ is closed under $\lambda$-filtered colimits in $\cv^\to$, then every ordinary $\lambda$-pure morphism in a locally $\lambda$-presentable $\ck$ is $(\lambda,\ce)$-pure.
\end{coro}
\begin{proof}
	Let $f$ be an ordinary $\lambda$-pure morphism in $\ck_0$. Then we can write $f$ as a $\lambda$-filtered colimit of split monomorphisms in $\ck^\to$ by \cite[Proposition~2.30]{AR}. Now, every split monomorphism is $(\lambda,\ce)$-pure by Proposition~\ref{cancell1} since identities are always $(\lambda,\ce)$-pure. Thus $f$ itself is  $(\lambda,\ce)$-pure by Proposition~\ref{filt-col-pure1} above.
\end{proof}

Next we study preservation of purity by $\cv$-functors. 

\begin{propo}\label{right-adj}
	Any right adjoint $\cv$-functor $F\colon\ck\to\cl$ preserving $\lambda$-filtered colimits sends $(\lambda,\ce)$-pure morphisms in $\ck$ to $(\lambda,\ce)$-pure morphisms in $\cl$.
\end{propo}
\begin{proof}
	Let $L\colon\cl\to\ck$ be the left adjoint to $F$; since $F$ preserves $\lambda$-filtered colimits then $L$ preserves the $\lambda$-presentable objects.
	Consider now a $(\lambda,\ce)$-pure morphism $f\colon K\to L$ in $\ck$ and any morphism $g\colon A\to B$ in $\cl$ between $\lambda$-presentable objects. Then $\cl (g,FL)\cong \ck(Lg,L)$ and $\cl (A,Ff)\cong \ck(LA,f)$, so that
	$$\cp(g,FL)\cong \cp(Lg,L)\ \text{ and }\ \cp(g,Ff)\cong \cp(Lg,f) $$ 
as well as $r_{g,Ff}\cong r_{Lg,f}$ in $\cv^\to$, where indices denote to which pair of morphisms Notation~\ref{strong-pure} applies. 
Now, since $A$ and $B$ are $\lambda$-presentable in $\cl$, the morphism $Lg$ has $\lambda$-presentable domain and codomain in $\ck$; thus we can apply the $(\lambda,\ce)$-purity of $f$ in $\ck$ to deduce that $ r_{Lg,f}$ lies in $\ce$. Since that is isomorphic to $r_{g,Ff}$, the $(\lambda,\ce)$-purity of $Ff$ follows.
\end{proof}

As a corollary we obtain that homming out of $\lambda$-presentable objects, as well as taking powers by them, preserves $(\lambda,\ce)$-purity:

\begin{coro}\label{representablypure1}
	Let $f\colon K\to L$ be a $(\lambda,\ce)$-pure morphism in a $\cv$-category $\ck$ with copowers; then:\begin{enumerate}
		\item $\ck(C,f)$ is $(\lambda,\ce)$-pure in $\cv$ for any $C\in\ck_\lambda$;
		\item if $\ck$ also has powers, $C\pitchfork f$ is $(\lambda,\ce)$-pure in $\ck$ for any $C\in\cv_\lambda$.
	\end{enumerate}
\end{coro}
\begin{proof}
	(1) It follows by Proposition~\ref{right-adj} above since $\ck(C,-)$ preserves $\lambda$-filtered colimits and has a left adjoint $(-)\cdot C\colon\cv\to\ck $.
	
	(2) Similarly, the $\cv$-functor $C\pitchfork(-)\colon\ck\to\ck$ preserves $\lambda$-filtered colimits (since $C$ is $\lambda$-presentable) and has a left adjoint $C\cdot (-)\colon\ck\to\ck$. Thus purity is preserved thanks to Proposition~\ref{right-adj} above.
\end{proof}

In the proposition below we denote by $\ce_\lambda$ the set of arrows in $\ce$ whose domain and codomain are $\lambda$-presentable in $\cv$, and by $\ce_\lambda^\perp$ the class of morphisms which are right orthogonal to every morphism in $\ce_\lambda$. 

\begin{propo}\label{good1}
	Assume that $\ce_\lambda^\perp=\cm$ and that $[e,X]\in\cm$ whenever $e\in\ce$ and $X\in\cv$. Then for any $(\lambda,\ce)$-pure $f\colon K\to L$ in a locally $\lambda$-presentable $\cv$-category $\ck$, we have that $\ck(C,f)\in\cm$ for any $C\in\ck$.
\end{propo}
\begin{proof}
	Since $\cm$ is closed under limits in $\cv^\to$, it is enough to prove the result for $C\in\ck_\lambda$. Furthermore, by Corollary~\ref{representablypure1} above, it is enough to show that any $(\lambda,\ce)$-pure morphism $f\colon K\to L$ in $\cv$ lies in $\cm$. 
	
	By the hypotheses on $\cm$, to prove that $f\in\cm$ it is enough to show that it is right orthogonal to every morphism in $\ce_\lambda$; for that it suffices to prove that the square below is a pullback.
	\begin{center}
		\begin{tikzpicture}[baseline=(current  bounding  box.south), scale=2]
			
			\node (a'0) at (0,0.9) {$[B,K]$};
			\node (b0) at (1.3,0.9) {$[B,L]$};
			\node (c0) at (0,0) {$[A,K]$};
			\node (d0) at (1.3,0) {$[A,L]$};
			
			\path[font=\scriptsize]

			(a'0) edge [->] node [above] {$[B,f]$} (b0)
			(b0) edge [->] node [right] {$[e,L]$} (d0)
			(c0) edge [->] node [below] {$[A,f]$} (d0)
			(a'0) edge [->] node [left] {$[e,K]$} (c0);
		\end{tikzpicture}	
	\end{center}
	Since $f$ is $(\lambda,\ce)$-pure, the $(\ce,\cm)$-factorization of $[e,K]$ is given by $\cp(e,f)$. But $[e,K]$ is in $\cm$ by hypothesis, then $[B,K]\cong \cp(e,f)$. Similarly, since $[e,L]$ is in $\cm$ by hypothesis, $[B,L]\cong\cp(g,L)$. It then follows by the definition of $\cp(g,f)$ that $[B,K]$ is the pullback of $[A,f]$ along $[g,L]$.
\end{proof}

\begin{rem}
	The equality $\ce_\lambda^\perp=\cm$ holds whenever the closure of $\ce_\lambda$ in $\ce$ under $\lambda$-filtered colimits is $\ce$ itself (because $\ce$ is closed under colimits). That is the case for example when $\cm$ is closed under $\lambda$-filtered colimits in $\cv^\to$ and $\cv_\lambda$ is closed under $(\ce,\cm)$-factorizations in $\cv$.
\end{rem}

\begin{rem}$ $
	{\setlength{\leftmargini}{1.6em}
	\begin{enumerate}
		
	\item If $\cv$ is a regular locally $\lambda$-presentable category and $\cv_\lambda$ is regularly embedded in $\cv$, the assumptions of \ref{good1} hold for the (regular epi, mono) factorization system since the class of monomorphisms are closed under $\lambda$-filtered colimits in $\cv^\to$ and $[e,X]$ is a monomorphisms whenever $e$ is a (regular) epimorphism.
	
	\item  The assumptions of \ref{good1} are also valid in $\CMet$ for $\ce=$ dense. The fact that $(\omega,\ce)$-pure morphisms are isometries, which follows from \ref{good1}, was proved in \cite[4.11]{R}.
	
	\item 	Proposition~\ref{good1} above does not hold in general for any factorization system. If we take $\ce=\ck^{\to}$ and $\cm=\ck^{\cong}$(see \ref{ex1}(2), then any morphism is $(\lambda,\ce)$-pure but $\ck(C,f)$ is an isomorphism for any $C\in\ck$ if and only if $f$ is. This factorization system satisfies $\ce_\lambda^\perp=\cm$, but not the requirement that $[e,X]\in\cm$ for any $e\in\ce$ and $X\in\cv$.
\end{enumerate}}
\end{rem}

Below, following \cite{LR12}, we say that a limit of some shape is {\em $\ce$-stable} if $\ce$ is closed under such limits in $\cv^\to$.

\begin{propo}\label{E-stable}
	The class of $(\lambda,\ce)$-pure morphisms in a $\cv$-category $\ck$ is stable under $\ce$-stable (conical) limits in $\ck^\to$. 
\end{propo}
\begin{proof}
	Assume that $f\colon K\to L$ is the $\ce$-stable limit of some $(\lambda,\ce)$-pure morphisms $f_i\colon K_i\to L_i$ in $\ck^\to$. For any $g\colon A\to B$ between $\lambda$-presentable objects, we can consider the diagram of Notation~\ref{strong-pure} applied to $f_i$ and $g$:
	\begin{center}
		\begin{tikzpicture}[baseline=(current  bounding  box.south), scale=2]
			
			\node (a0) at (-0.5,1.2) {$\ck(B,K_i)$};
			\node (b0) at (1,1.2) {$\ck(B,L_i)$};
			\node (c0) at (0,-0.2) {$\ck(A,K_i)$};
			\node (d0) at (1.5,-0.2) {$\ck(A,L_i)$};
			\node (b'0) at (0,0.55) {$\cp(g,f_i)$};
			\node (c'0) at (1.5,0.55) {$\cp(g,L_i)$};
			
			\path[font=\scriptsize]

			(a0) edge [->] node [above] {$\ck(B,f_i)$} (b0)
			(b'0) edge [->] node [above] {$r'_i$} (c'0)
			(b0) edge [->>] node [right] {} (c'0)
			(c'0) edge [>->] node [right] {} (d0)
			(c0) edge [->] node [below] {$\ck(A,f_i)$} (d0)
			(a0) edge [bend right,->] node [left] {$\ck(g,K_i)$} (c0)
			(a0) edge [->>] node [right] {$r_i$} (b'0)
			(b'0) edge [>->] node [right] {} (c0);
		\end{tikzpicture}	
	\end{center}
	Since the factorization system $(\ce,\cm)$ is functorial, $\ce$ and $\cm$ are stable under $\ce$-stable limits, pullbacks commute with all limits, and the representables preserve all limits, it follows that the componentwise limit of the diagram above in $\cv$, is exactly the diagram displayed in Notation~\ref{strong-pure} for $f$ and $g$. It follows in particular that 
	$$r\colon \ck(B,K)\to\cp(g,f)$$
	is the $\ce$-stable limit of the $r_i$'s. But each $r_i$ lies in $\ce$, by purity of $f_i$; thus also $r$ lies in $\ce$ and $f$ is $(\lambda,\ce)$-pure.
\end{proof}

Recall that, given a collection $A_i$, $i\in I$ of objects of a locally
presentable category $\ck$ and an ultrafilter $\cu$ on the set $I$, then
the ultraproduct $\prod_\cu A_i$ is defined as the colimit of the directed diagram
induced by the products $\prod_{i\in U} A_i$, where $U$ ranges through $\cu$, and
the projections $\prod_{i\in U}A_i\to\prod_{i\in U'}A_i$, where $U'\subseteq U$ (see, e.g., \cite{AR}).

\begin{coro}
	Let $(\ce,\cm)$ be proper and $\ce$ and $\cm$ be respectively closed under products and filtered colimits in $\cv^\to$. Then in a locally finitely presentable $\cv$-category $\ck$:\begin{enumerate}
		\item the class of $(\omega,\ce)$-pure morphisms is closed under ultraproducts;
		\item for any $K\in\ck$ the inclusion $d_K\colon K\to K^\cu$, of $K$ into its ultrapower by an ultrafilter $\cu$, is $(\omega,\ce)$-pure;
		\item $f\colon K\to L$ is $(\omega,\ce)$-pure if and only if the ultrapower $f^\cu\colon K^\cu\to L^\cu$ is  $(\omega,\ce)$-pure, for some ultrafilter $\cu$.
	\end{enumerate} 
\end{coro}
\begin{proof}
	$(1)$. The ultraproduct of a family of morphisms $(f_i)_i$ in $\ck$ can be computed as a filtered colimit of products of subfamilies of the $(f_i)_i$. Thus it is enough to apply Proposition~\ref{E-stable} above and Proposition~\ref{filt-col-pure1}.
	
	$(2)$. The inclusion $d_K\colon K\to K^\cu$ is a filtered colimit of split monomorphisms, thus it is pure in the ordinary sense. By Corollary~\ref{pure-Epure} the inclusion is also $(\omega,\ce)$-pure.
	
	$(3)$. If $f\colon K\to L$ is $(\omega,\ce)$-pure then so is $f^\cu$ by point $(1)$. Conversely, if $f^\cu$ is $(\omega,\ce)$-pure then also the composite $f^\cu d_K$ is by point $(2)$ and Proposition~\ref{composition1}. But $f^\cu d_K=d_L f$; therefore $f$ is $(\omega,\ce)$-pure by Proposition~\ref{cancell1}.
\end{proof}

Next we recall the notion of $\ce$-injectivity class from \cite[Section~2]{LR12}: given a morphism $f\colon A\to B$ in a $\cv$-category $\ck$, we say that $X\in\ck$ is {\em $f$-injective over $\ce$} if the map
$$ \ck(f,X)\colon\ck(B,X)\longrightarrow\ck(A,X) $$
lies in $\ce$. Given a set of morphisms $\ch$ in $\ck$, we denote by $\tx{Inj}_\ce \ch$ the full subcategory of $\ck$ spanned by the object which are $f$-injective over $\ce$ for any $f\in\ch$. We call {\em $\ce$-injectivity class} any full subcategory of $\ck$ arising that way; if the morphisms in $\ch$ have $\lambda$-presentable domain and codomain, we call $\tx{Inj}_\ce \ch$ a {\em $(\lambda,\ce)$-injectivity class}. 

Then we can prove the following result. Note that below we use the term ``subobject'' when referring to the closure property under $(\lambda,\ce)$-pure morphism even though such morphisms need not be monomorphic in general; we do this since that is the standard way of phrasing the closure property. 

\begin{propo}\label{closure1}
	Every $(\lambda,\ce)$-injectivity class $\tx{Inj}_\ce \ch$ in a complete and cocomplete $\cv$-category $\ck$ is closed under:\begin{enumerate}
		\item $\ce$-stable limits;
		\item $\lambda$-filtered colimits;
		\item $(\lambda,\ce)$-pure subobjects; that is, if $f\colon K\to L$ is $(\lambda,\ce)$-pure and $L\in \tx{Inj}_\ce \ch$, then also $K\in \tx{Inj}_\ce \ch$.
	\end{enumerate}
\end{propo}
\begin{proof}
	(1) is given by \cite[Proposition~2.6]{LR12} and (2) by \cite[Proposition~2.7]{LR12} since $\ce$, being part of a factorization system, is closed in $\cv^\to$ under all conical colimits.
	
	(3) Take a $(\lambda,\ce)$-pure morphism $f\colon K\to L$ with $L\in\tx{Inj}_\ce \ch$. Let $g\colon A\to B$ be in $\ch$, which is then between $\lambda$-presentable objects.
	
	By $\ce$-injectivity of $L$ with respect to $g$, the map $\ck(g,L)$ is in $\ce$, so that $\cp(g,L)\cong \ck(A,L)$. It follows that $\cp(g,f)\cong \ck(A,K)$, it being the pullback of $\cp(g,L)$ along $\ck(A,f)$. Therefore the map $r$ appearing in the definition of $(\lambda,\ce)$-purity (Notation~\ref{strong-pure}) coincides with $$\ck(g,K)\colon\ck(B,K)\to\ck(A,K).$$
	Since $f$ is a $(\lambda,\ce)$-pure morphism, $\ck(g,K)$ lies in $\ce$ and hence $K$ is $\ce$-injective with respect to $g$.
\end{proof}

\begin{rem} 
	Point (3) of Proposition \ref{closure1} was proved in $\CMet$ as \cite[6.4]{R}.
\end{rem}

\begin{rem}\label{cond} 
	Suppose that $(\ce,\cm)$ is proper, $\cm$ is closed under $\lambda$-filtered colimits in $\cv^\to$, and that products are $\ce$-stable. Then every $(\lambda,\ce)$-injectivity class $\cl$ in a locally $\lambda$-presentable $\cv$-category $\ck$ is an ordinary $\lambda$-injectivity class.
	
	Indeed, by the proposition above $\cl$ is closed under products, $\lambda$-filtered colimits and $\lambda$-pure subobjects (see \ref{pure-E-pure}). Following \cite{RAB}, it is then a $\lambda$-injectivity class.
\end{rem}

\begin{rem}\label{basechange}
	Let $\cw=(\cw_0,I_\cw,\otimes_\cw)$ be locally $\lambda$-presentable as closed category, and $\cv$ be a locally $\lambda$-presentable $\cw$-category. Assume moreover that $\cv_0$ is endowed with an enriched symmetric monoidal closed structure $(\cv,I_\cv,\otimes_\cv)$ which makes it locally $\lambda$-presentable as a closed category. Then the $\cw$-functor $$U:=\cv(I_\cv,-)\colon\cv\to \cw$$ is continuous and preserves $\lambda$-filtered colimits; by local presentability of $\cv$ and $\cw$, it has a left adjoint $F\colon\cw\to\cv$ which is strong monoidal (\cite[2.1]{Kel73}).
	
	Under these assumptions there is a change of base functor $$U_*\colon\cv\tx{-}\bo{Cat}\to\cw\tx{-}\bo{Cat},$$ which sends a $\cv$-category $\ck$ to the $\cw$-category $U_*\ck$ with the same objects as $\ck$ and homs $$(U_*\ck)(X,Y):=U\ck(X,Y),$$ so that $\ck_0=(U_*\ck)_0$ and hence morphisms in $\ck$ are the same as morphisms in $U_*\ck$. See for instance \cite[Section~6]{EK}.
	
	It is now easy to see that the following properties hold whenever $\cv$ and $\cw$ are equipped with factorization systems $(\ce_\cv,\cm_\cv)$ and $(\ce_\cw,\cm_\cw)$ respectively: \begin{enumerate}
		\item if $U(\ce_\cv)\subseteq\ce_\cw$ and $U(\cm_\cv)\subseteq\cm_\cw$, then $\ce_\cv$-purity with respect to $g$ in $\ck$ implies $\ce_\cw$-purity with respect to $g$ in $U_*\ck$;
		\item if $\ce_\cv\supseteq U^{-1}(\ce_\cw)$ and $\cm_\cv\supseteq U^{-1}(\cm_\cw)$, then  $\ce_\cw$-purity with respect to $g$ in $U_*\ck$ implies $\ce_\cv$-purity with respect to $g$ in $\ck$;
		\item if $\ce_\cv=U^{-1}\ce_\cw$ and $\cm_\cv=U^{-1}\cm_\cw$, then $\ce_\cv$-purity with respect to $g$ in $\ck$ is equivalent to $\ce_\cw$-purity with respect to $g$ in $U_*\ck$.
	\end{enumerate}
	This follows directly from the definition of $(\lambda,\ce)$-purity (since $U$ preserves pullbacks) by applying $U$ to the diagram of Notation~\ref{strong-pure}. We leave the details to the reader as we will not makes use this result in any of our Theorems.\\
	Finally, since $U$ preserves $\lambda$-filtered colimits, the $\lambda$-presentable objects of $\ck$ and $U_*\ck$ coincide; thus, under the hypothesis of (3) above, a morphism is $(\lambda,\ce_\cv)$-pure in $\ck$ if and only if it is $(\lambda,\ce_\cw)$-pure in $U_*\ck$.
\end{rem}

\begin{exam}\label{Ban1}
	A direct application of the result above is obtained by considering $\Ban$ as a $\CMet$-category (see \cite[Example~4.5(3)]{AR1}); the induced $\CMet$-functor $U$ is the unit-ball functor
	$U\colon\Ban \to \Met$.
 	Moreover, the (dense, isometry) factorization system on $\Ban$ can be obtained, as in (3) above, by taking the preimage of the factorization system (dense, isometry) in $\CMet$.
	
	It follows that a morphism $f$ in a $\Ban$-category $\ck$ is $(\lambda,\tx{dense})$-pure if and only if it is so in the $\CMet$-category $U_*\ck$.
\end{exam}

\begin{rem}
	Following the nomenclature of \cite{LT22} for accessible $\cv$-categories, we know that an  accessibly embedded full subcategory of a locally presentable $\cv$-category is accessible if and only if it is closed under ordinary $\lambda$-pure subobjects for some $\lambda$ (see \cite[3.23]{LT22} and \cite[2.36]{AR}). The same holds with respect to conical accessibility.
	
	Let us consider now enriched purity. Under the hypothesis of Corollary~\ref{pure-Epure} and using Proposition~\ref{closure1} above, any accessibly embedded full subcategory of a locally presentable $\cv$-category that is closed under $(\lambda,\ce)$-pure subobjects, for some $\lambda$, is accessible. We do not know, however, whether the converse implication holds; that is, if every accessible and accessibly embedded $\cv$-category in a locally presentable $\cv$-category is closed under enriched pure subobjects.
\end{rem}

\section{A weaker notion}\label{enr-purity2}
The following notation is needed for our second notion of purity.
\begin{nota}\label{notation}
Let $\ck$ be a $\cv$-category and $(\ce,\cm)$ a proper factorization system  in $\cv$. Consider morphisms $f\colon K\to L$ and $g\colon A\to B$ in $\ck$.  Let $\cq(g,f)$ be the $(\ce,\cm)$ factorization 
\begin{center}
	\begin{tikzpicture}[baseline=(current  bounding  box.south), scale=2]
		
		\node (21) at (0,0) {$\ck^\to(g,f)$};
		\node (22) at (1.2,0) {$\cq(g,f)$};
		\node (23) at (2.4,0) {$\ck(A,K)$};
		
		\path[font=\scriptsize]
		
		(21) edge [->>] node [above] {$p'$} (22)
		(22) edge [>->] node [above] {$p''$} (23);
	\end{tikzpicture}	
\end{center} 
of the projection $p:\ck^\to(g,f)\to \ck(A,K)$ sending $(u,v)$ to $u$. 
Let
$$
q':\ck(B,K)\to\ck^\to(g,f)
$$
be induced by the universal property of the pullback defining $\ck^\to(g,f)$ applied to $\ck(g,K)$ and $\ck(B,f)$. Explicitly, this sends $t:B\to K$ to $(tg,ft):g\to f$. Let $q=p'q':\ck(B,K)\to\cq(g,f)$ so that we have the commutative diagram below.
\begin{center}
	\begin{tikzpicture}[baseline=(current  bounding  box.south), scale=2]
		
		\node (a0) at (0,1.2) {$\ck^\to(g,f)$};
		\node (a'0) at (-1.5,1.5) {$\ck(B,K)$};
		\node (b0) at (1.5,1.2) {$\ck(B,L)$};
		\node (c0) at (0,-0.2) {$\ck(A,K)$};
		\node (d0) at (1.5,-0.2) {$\ck(A,L)$};
		\node (b'0) at (-0.5,0.5) {$\cq(g,f)$};
		
		\path[font=\scriptsize]
		
		(a0) edge [<-] node [above] {$q'$} (a'0)
		(a'0) edge [bend left,->] node [above] {$\ck(B,f)$} (b0)
		(a'0) edge [bend right,->] node [left] {$\ck(g,K)$} (c0)
		(a0) edge [->] node [right] {$p$} (c0)
		(a0) edge [->] node [left] {} (b0)
		(b0) edge [->] node [right] {$\ \ \ck(g,L)$} (d0)
		(c0) edge [->] node [below] {$\ck(A,f)$} (d0)
		(a0) edge [->>] node [left] {$p'\ $} (b'0)
		(a'0) edge [dashed, ->] node [left] {$q\ $} (b'0)
		(b'0) edge [>->] node [left] {$p''\ $} (c0);
	\end{tikzpicture}	
\end{center}
\end{nota}

\begin{defi}\label{barely-pure-def}
We will say that $f\colon K\to L$ is \textit{barely} $\ce$-\textit{pure with respect to $g$} if the map $q$ above is in $\ce$. We say that $f$ is \textit{barely}  $(\lambda,\ce)$-\textit{pure} if it is is barely $(\lambda,\ce)$-pure with respect to every $g:A\to B$ with $A$ and $B$ $\lambda$-presentable.
\end{defi}

\begin{exam}\label{ex2}
When $\cv=\CMet$ and we take the factorization system (dense, isometry) (see 
\cite[3.16(2)]{AR1}). Then barely $(\lambda,\ce)$-pure morphisms coincide with the barely $\lambda$-pure morphisms from \cite{RT}. Indeed, $\cq(f,g)$ consists of those $u:A\to K$ such that for every $\eps>0$ there are $u':A\to K$ and $v:B\to L$ such that $u\sim_{\eps}u'$ and $fu'=vg$. Clearly, if $f$ is $(\lambda,\ce)$-pure then it is barely $\lambda$-ap-pure. Conversely, let $f$ be barely $\lambda$-ap-pure and consider $u\in\cq(f,g)$ and $\eps>0$. There are $u':A\to K$ and $v:B\to L$ such that $u\sim_{\frac{\eps}{2}} u'$ and $fu'=vg$. There is $t:B\to A$ such that $tg\sim_{\frac{\eps}{2}}u'$. Hence $tg\sim_{\eps}u$.
\end{exam}

\begin{lemma}\label{pure-bpure}
	If $(\ce,\cm)$ is proper, every $(\lambda,\ce)$-pure morphism is barely $(\lambda,\ce)$-pure. Furthermore, if $\ce$ is stable under pullbacks in $\cv$, then $(\lambda,\ce)$-pure and barely $(\lambda,\ce)$-pure morphisms coincide.
\end{lemma}
\begin{proof}
	Let $f\colon K\to L$ be a $(\lambda,\ce)$-pure morphism in $\ck$ and $g\colon A\to B$ a morphism in $\ck_\lambda$. In the notation of \ref{notation}, the universal property of the pullback defining $\cp(g,f)$ induces a map $t\colon\ck^\to(g,f)\to\cp(g,f)$ making the square below commute
	\begin{center}
		\begin{tikzpicture}[baseline=(current  bounding  box.south), scale=2]
			
			\node (a0) at (0,0.9) {$\ck^\to(g,f)$};
			\node (b0) at (1.3,0.9) {$\cq(g,f)$};
			\node (c0) at (0,0) {$\cp(g,f)$};
			\node (d0) at (1.3,0) {$\ck(A,K)$};
			
			\path[font=\scriptsize]

			(a0) edge [->>] node [above] {$p'$} (b0)
			(b0) edge [>->] node [right] {} (d0)
			(b0) edge [dashed,>->] node [below] {$\ \ s$} (c0)
			(c0) edge [>->] node [below] {} (d0)
			(a0) edge [->] node [left] {$t$} (c0);
		\end{tikzpicture}	
	\end{center}
	and, since $p'\in\ce$ there exists $s$ in $\cm$ as dashed above making the triangles commute. By pre-composition with $q'$, we obtain that $sq=r$, and $r$ is in $\ce$ by $(\lambda,\ce)$-purity of $f$. Thus, if $(\ce,\cm)$ is proper, then $s$ is also in $\ce$ and hence is an isomorphism. Therefore $q$ is also in $\ce$ and $f$ is barely $(\lambda,\ce)$-pure.
	
	Assume now that $\ce$ is stable under pullbacks in $\cv$. The map $t$ can be seen as the pullback of the morphism $\ck(B,L)\twoheadrightarrow\cp(g,L)$ along $\cp(g,f)\to\cp(g,L)$ from Notation~\ref{strong-pure}; thus $t$ is in $\ce$ by hypothesis. The uniqueness of the $(\ce,\cm)$-factorization then implies that $s\colon\cq(g,f)\to\cp(g,f)$ above is an isomorphism. Thus $f$ is $(\lambda,\ce)$-pure if and only if it is barely $(\lambda,\ce)$-pure.
\end{proof}

\begin{rem}
	In $\Met$ and $\CMet$, dense maps are not stable under pullbacks. Indeed,
	consider the pullback
	\begin{center}
		\begin{tikzpicture}[baseline=(current  bounding  box.south), scale=2]
			
			\node (a'0) at (0,0.8) {$\emptyset$};
			\node (b0) at (1.4,0.8) {$\mathbb N$};
			\node (c0) at (0,0) {$\{0\}$};
			\node (d0) at (1.4,0) {$\{\frac{1}{n}|n\in \Bbb N\}\cup\{0\}$};
			
			\path[font=\scriptsize]

			(a'0) edge [->] node [above] {} (b0)
			(b0) edge [->] node [right] {$g$} (d0)
			(c0) edge [->] node [below] {} (d0)
			(a'0) edge [->] node [left] {} (c0);
		\end{tikzpicture}	
	\end{center}
	where $g$ is from \ref{ex1}(3) and is dense. But $\emptyset\to\{0\}$ is not dense.
	
	Nevertheless, $(\lambda,\ce)$-pure and barely $(\lambda,\ce)$-pure are equivalent notions by \ref{equ-purities}.
\end{rem}

\begin{propo}\label{shift}
If $\ce\subseteq\ce'$ and $(\ce',\cm')$ is proper, then every barely $(\lambda,\ce)$-pure morphism is barely $(\lambda,\ce')$-pure.
\end{propo}
\begin{proof}
We will follow Notation~\ref{notation} with indices denoting the factorization system to which they apply.
If $\ce\subseteq\ce'$ then $\cm'\subseteq\cm$ and thus there exists
$t:\cq_{\ce}(g,f)\to\cq_{\ce'}(g,f)$ such that $tp'_{\ce}=p'_{\ce'}$.
By properness of the factorization system, $t\in\ce'$ and thus
$$
q_{\ce'}=p'_{\ce'}q'=tp'_{\ce}q'=tq_{\ce}
$$
belongs to $\ce'$ (note that $q'$ does not depend on $\ce$ or $\ce'$).
\end{proof}

\begin{rem}
	The property above is one that doesn't seem to hold for the notion of purity considered in the previous section. Nonetheless, it will be satisfied whenever the two notions of purity coincide.
\end{rem}

Next we can prove some results in the same spirit as those of the previous section.

\begin{propo}\label{cancell}
If $(\ce,\cm)$ is proper then barely $(\lambda,\ce)$-pure morphisms are left-cancellable; that is, if $f_2f_1$ is barely $(\lambda,\ce)$-pure then $f_1$ is barely $(\lambda,\ce)$-pure.
\end{propo}
\begin{proof}
We follow Notation~\ref{notation}. Let $f_2f_1$ be $(\lambda,\ce)$-pure and consider the factorizations
$$
\ck^\to(g,f_i) \xrightarrow{\ p'_i\ }\cq(g,f_i) \xrightarrow{\ p''_i\ } \ck(A,K),
$$ 
for $i=1,2$, and
$$
\ck^\to(g,f) \xrightarrow{\ p'\ }\cq(g,f) \xrightarrow{\ p''\ } \ck(A,K)
$$ 
where $f=f_2f_1$. Since $(\id_K,f_2):f_1\to f$ defines a morphism in $\ck^\to$, we have a map
$$
\ck^\to(g,(\id_K,f_2))\colon\ck^\to(g,f_1)\to\ck^\to(g,f)
$$
for which $p \ck^\to(g,(\id_K,f_2))=p_1$. Thus there is an induced $s\colon\cq(g,f_1)\to\cq(g,f)$ such that 
$$
sp'_1=p'\ck^\to(g,(\id_K,f_2))
$$ 
and $p''s=p''_1$. Moreover, $sq_1=q$ and thus $s$ is in $\ce$. Since it is also in $\cm$, $s$ is an isomorphism. Hence $q_1$ is in $\ce$ and thus $f_1$ 
is $(\lambda,\ce)$-pure.
\end{proof}

\begin{propo}\label{filt-col-pure}
	If $\cm$ is closed under $\lambda$-filtered colimits in $\cv^\to$, then any $\lambda$-filtered colimit of barely $(\lambda,\ce)$-pure morphisms in $\ck^\to$ is barely $(\lambda,\ce)$-pure.
\end{propo}
\begin{proof}
	Let $f\colon K\to L$ be a $\lambda$-filtered colimit in $\ck^\to$ of some barely $(\lambda,\ce)$-pure morphisms $f_i\colon K_i\to L_i$, and consider any $g\colon A\to B$ in $\ck_\lambda$. Since $A\in\ck_\lambda$ and $g$ is a $\lambda$-presentable object of $\ck^\to$, the projection $p\colon \ck^\to(g,f)\to \ck(A,K_i)$ is the $\lambda$-filtered colimit of the projections $p_i\colon \ck^\to(g,f_i)\to \ck(A,K)$. Now, we obtain an induced $\lambda$-filtered diagram on the $\cq(g,f_i)$ whose colimit coincides with $\cq(g,f)$, since $\ce$ and $\cm$ are closed under $\lambda$-filtered colimits in $\ck^\to$. It follows that the map $q\colon \ck(B,K)\to \cq(g,f)$ is the colimit of the maps $q_i\colon \ck(B,K_i)\to \cq(g,f_i)$, which lie in $\ce$ since every $f_i$ is $(\lambda,\ce)$-pure. Thus $q$ also lies in $\ce$ and $f$ is barely $(\lambda,\ce)$-pure.
\end{proof}

Next we study preservation of bare $(\lambda,\ce)$-purity by $\cv$-functors:

\begin{propo}\label{bare-rightadj}
	Any right adjoint $\cv$-functor $F\colon\ck\to\cl$ preserving $\lambda$-filtered colimits sends barely $(\lambda,\ce)$-pure morphisms in $\ck$ to barely $(\lambda,\ce)$-pure morphisms in $\cl$.
\end{propo}
\begin{proof}
	The proof is analogue to that of Proposition~\ref{right-adj}.
\end{proof}

As a corollary we obtain that homming out of $\lambda$-presentable objects, as well as taking powers by them, preserves bare $(\lambda,\ce)$-purity:

\begin{coro}\label{representablypure}
	Let $f\colon K\to L$ be a barely $(\lambda,\ce)$-pure morphism in a $\cv$-category $\ck$ with copowers; then:\begin{enumerate}
		\item $\ck(C,f)$ is barely $(\lambda,\ce)$-pure in $\cv$ for any $C\in\ck_\lambda$;
		\item if $\ck$ also has powers, $C\pitchfork f$ is barely $(\lambda,\ce)$-pure in $\ck$ for any $C\in\cv_\lambda$.
	\end{enumerate}
\end{coro}
\begin{proof}
	Same arguments of Corollary~\ref{representablypure1}, now relying on Proposition~\ref{bare-rightadj} above.
\end{proof}

In the statement below we adopt the same notation of Proposition~\ref{good1}.

\begin{propo}\label{good}
	Assume that $\ce_\lambda^\perp=\cm$ and that $[e,X]\in\cm$ whenever $e\in\ce$ and $X\in\cv$. Then for any barely $(\lambda,\ce)$-pure $f\colon K\to L$ in a locally $\lambda$-presentable $\cv$-category $\ck$, we have that $\ck(C,f)\in\cm$ for any $C\in\ck$.
\end{propo}
\begin{proof}
	Since $\cm$ is closed under limits in $\cv^\to$, it is enough to prove the result for $C\in\ck_\lambda$. Furthermore, by Corollary~\ref{representablypure} above, it is enough to show that any barely $(\lambda,\ce)$-pure morphism $f\colon K\to L$ in $\cv$ lies in $\cm$. 
	
	By the hypotheses on $\cm$, to prove that $f\in\cm$ it is enough to show that it is right orthogonal to every morphism in $\ce_\lambda$; for that it suffices to prove that the square below is a pullback.
	\begin{center}
		\begin{tikzpicture}[baseline=(current  bounding  box.south), scale=2]
			
			\node (a'0) at (0,0.9) {$[B,K]$};
			\node (b0) at (1.3,0.9) {$[B,L]$};
			\node (c0) at (0,0) {$[A,K]$};
			\node (d0) at (1.3,0) {$[A,L]$};
			
			\path[font=\scriptsize]

			(a'0) edge [->] node [above] {$[B,f]$} (b0)
			(b0) edge [->] node [right] {$[e,L]$} (d0)
			(c0) edge [->] node [below] {$[A,f]$} (d0)
			(a'0) edge [->] node [left] {$[e,K]$} (c0);
		\end{tikzpicture}	
	\end{center}
	Note that $[e,K]$ is in $\cm$ by hypothesis. Moreover $[e,K]=p''q$ with $q\in\ce$ and $p''\in\cm$. It follows that $q\in\cm$ and hence is an isomorphism. On the other hand, $p\colon \cv^\to(e,f)\to[A,K] $ is obtained by pulling back $[e,L]$ which, again by hypothesis, is in $\cm$. Therefore $p\in\cm$ as well, and $p'$ is then an isomorphism. Thus $q'\colon [B,K]\to \cv^\to(e,f)$ is an isomorphism as well. This shows that the square above is a pullback, and hence that $f\in\cm$.
\end{proof}

\begin{rem}
	If the factorization system $(\ce,\cm)$ is proper, then Proposition~\ref{good1} follows by the result above plus Lemma~\ref{pure-bpure}.
\end{rem}

\begin{defi}\label{e-split}
	We say that a morphism $s\colon K\to L$ in a $\cv$-category $\ck$ is an {\em $\ce$-split monomorphism} if the map $\ck(s,K)\colon \ck(L,K)\to\ck(K,K)$ is in $\ce$.
\end{defi}

\begin{rem}
	The notion of $\ce$-split monomorphism should not be confused with that of $\ce$-split epimorphism; this is a map $s$ for which $\ck(L,s)\colon \ck(L,K)\to\ck(L,L)$ is required to be in $\ce$.
\end{rem}

\begin{exam}$ $\begin{enumerate}
		\item If $\ce$ is the class of surjections in the sense of \ref{surjection1}
		then $\ce$-split monomorphisms are precisely the usual split monomorphisms.
		\item For $\Met$ (or $\CMet$) and $\ce=$ dense, $\ce$-split monomorphisms coincide with ap-split monomorphisms from \cite{RT}.
	\end{enumerate}
\end{exam}

\begin{rem}\label{relative}$ $\begin{enumerate}

	\item  An $\ce$-split morphism $s:K\to L$ is barely $\ce$-pure with respect to itself. Indeed, since $pq'=\ck(s,K)$, we get that $p''$ is in $\ce$, hence it is an isomorphism. Thus $q\in\ce$.

	\item  The proof of \ref{filt-col-pure} yields that, if $\cm$ is closed under $\lambda$-filtered colimits in $\cv^\to$, then any $\lambda$-filtered colimit of morphisms barely $\ce$-pure with respect to $g$ is barely $\ce$-pure with respect to $g$.

	\item  Assume that $f$ is barely $\ce$-pure with respect to $g$ and $g'\to g$ is a morphism in $\ck^\to$. We do not know whether $f$ is then barely $\ce$-pure with respect to $g'$. It is valid in all our examples and it implies that $\ce$-split morphisms are barely $\ce$-pure.
\end{enumerate}
\end{rem}

As in the previous section, next we talk about $\ce$-injectivity classes.

\begin{propo}\label{closure}
	If barely $(\lambda,\ce)$-pure and $(\lambda,\ce)$-pure morphisms coincide, then every $(\lambda,\ce)$-injectivity class is closed under barely $(\lambda,\ce)$-pure subobjects.
\end{propo}
\begin{proof}
	Follows from Proposition~\ref{closure1}.
\end{proof}


\begin{exam}\label{Ban2}
In $\CMet$ with the (dense, isometry) factorization system, a morphism is $(\lambda,\ce)$-pure if and only if it is barely $(\lambda,\ce)$-pure (see \cite[Lemma~4.3]{R} or Proposition~\ref{equ-purities}). Therefore, by the remark above plus Example~\ref{Ban1}, it follows that also when enriching over $\Ban$, with the (dense, isometry) factorization, the notions of $(\lambda,\ce)$-purity and bare $(\lambda,\ce)$-purity coincide.
\end{exam}

\section{The characterization theorem}\label{E-is-inj}
In this section, we will characterize $(\lambda,\ce)$-injectivity classes.
To achieve this, we will need the following assumptions about $\ce$.

\begin{assume}\label{assumption}$ $
	\begin{itemize}
		\item $(\ce,\cm)$ is an enriched and proper factorization system;
		\item $\cm$ is closed under $\lambda$-filtered colimits in $\cv^\to$;
		\item $\ce$ can be written as an injectivity class in $\cv^\to$ with respect to maps $(!_Z,y)$ as below
		\begin{center}
			\begin{tikzpicture}[baseline=(current  bounding  box.south), scale=2]
				
				\node (a'0) at (0,0.8) {$0$};
				\node (b0) at (1.1,0.8) {$Z$};
				\node (c0) at (0,0) {$X$};
				\node (d0) at (1.1,0) {$Y$};
				
				\path[font=\scriptsize]

				(a'0) edge [->] node [above] {$!_Z$} (b0)
				(b0) edge [->] node [right] {$d$} (d0)
				(c0) edge [->] node [below] {$y$} (d0)
				(a'0) edge [->] node [left] {$!_X$} (c0);
			\end{tikzpicture}	
		\end{center}
		with $X\in\ck_\lambda$.
	\end{itemize}
\end{assume}

\begin{exam}\label{ex-ass}$ $
	{\setlength{\leftmargini}{1.6em}
	\begin{enumerate}
		\item Surjections in $\cv$ can be characterized as those maps which are injective with respect to
		\begin{center}
			\begin{tikzpicture}[baseline=(current  bounding  box.south), scale=2]
				
				\node (a'0) at (0,0.8) {$0$};
				\node (b0) at (1.1,0.8) {$I$};
				\node (c0) at (0,0) {$I$};
				\node (d0) at (1.1,0) {$I$};
				
				\path[font=\scriptsize]

				(a'0) edge [->] node [above] {$!$} (b0)
				(b0) edge [->] node [right] {$d=1$} (d0)
				(c0) edge [->] node [below] {$y=1$} (d0)
				(a'0) edge [->] node [left] {$!$} (c0);
			\end{tikzpicture}	
		\end{center}
		Indeed, for a morphism $f\colon A\to B$, to give a map $!\to f$ in $\cv^\to$ is the same as giving an element $b\in\cv_0(I,B)$; then such $b$ factorizes through the square above if and only if there exists $a\in\cv_0(I,A)$ with $fa=b$.
		\item Dense maps in $\CMet^\to$ and $\Met^\to$ can be characterized as those maps which are injective with respect to the family
		\begin{center}
			\begin{tikzpicture}[baseline=(current  bounding  box.south), scale=2]
				
				\node (a'0) at (0,0.8) {$0$};
				\node (b0) at (1.1,0.8) {$1$};
				\node (c0) at (0,0) {$1$};
				\node (d0) at (1.1,0) {$2_\epsilon$};
				
				\path[font=\scriptsize]

				(a'0) edge [->] node [above] {$!$} (b0)
				(b0) edge [->] node [right] {$d=i_1$} (d0)
				(c0) edge [->] node [below] {$y=i_0$} (d0)
				(a'0) edge [->] node [left] {$!$} (c0);
			\end{tikzpicture}	
		\end{center}
		for every $\epsilon>0$. Here $2_\epsilon$ is the metric space given by two points $\{0,1\}$ of distance $\epsilon$, and $i_0$ and $i_1$ are the two inclusions.\\
		Arguing as above. For any $f\colon A\to B$, a map $!\to f$ corresponds to an element $b\in B$, and such $b$ factorizes through the square above if and only if there exists $a\in A$ with $d(fa,b)<\epsilon$. Asking this for each $\epsilon>0$ is equivalent to requiring that $f$ is dense.
		
		\item Let $\Ban$ be the category of Banach spaces and linear maps of norm $\leq 1$. Dense maps in $\Ban^\to$ can be characterized as those maps which are injective with respect to the family
		\begin{center}
			\begin{tikzpicture}[baseline=(current  bounding  box.south), scale=2]
				
				\node (a'0) at (0,0.8) {$0$};
				\node (b0) at (1.1,0.8) {$\Bbb C$};
				\node (c0) at (0,0) {$\Bbb C$};
				\node (d0) at (1.1,0) {$F2_\epsilon$};
				
				\path[font=\scriptsize]

				(a'0) edge [->] node [above] {$!$} (b0)
				(b0) edge [->] node [right] {$d=Fi_1$} (d0)
				(c0) edge [->] node [below] {$y=Fi_0$} (d0)
				(a'0) edge [->] node [left] {$!$} (c0);
			\end{tikzpicture}	
		\end{center}
for every $\epsilon>0$. Here $F$ is the left adjoint to the unit ball
functor $U:\Ban\to\CMet$ (see \cite{R2}). This follows from the fact that
$f$ is dense in $\Ban$ if and only if $Uf$ is dense in $\CMet$. Recall that $F1=\Bbb C$.

\item Let $\omega$-$\CPO$ be the cartesian closed category of posets
with joins of non-empty $\omega$-chains and maps preserving joins of non-empty $\omega$-chains. A morphism $f:A\to B$ in  $\omega$-$\CPO$ is called {\em dense} if $B$ is
the closure of $f[A]$ under joins of $\omega$-chains; these form the left part of a factorization system (dense, embedding) -- see~\cite[A.7]{ADV}. Dense maps can be characterized as those morphism which are injective in  $\omega$-$\CPO^\to$ with respect to
\begin{center}
			\begin{tikzpicture}[baseline=(current  bounding  box.south), scale=2]
				
				\node (a'0) at (0,0.8) {$0$};
				\node (b0) at (1.1,0.8) {$\omega\cdot 1$};
				\node (c0) at (0,0) {$1$};
				\node (d0) at (1.1,0) {$\omega+1$};
				
				\path[font=\scriptsize]

				(a'0) edge [->] node [above] {$!$} (b0)
				(b0) edge [->] node [right] {$d$} (d0)
				(c0) edge [->] node [below] {$y$} (d0)
				(a'0) edge [->] node [left] {$!_1$} (c0);
			\end{tikzpicture}	
			\end{center}
where $\omega+1$ is the successor of $\omega$ seen as an ordinal (this is obtained from $\omega$ by freely adding the join of all elements), $y$ sends the only element of $1$ to the top element of $\omega+1$, and $d$ sends the countable anti-chain $\omega\cdot 1$ to the chain $\omega\subseteq\omega+1$.\\
To see this, consider $f\colon A\to B$; then (as before) a map $!_1\to f$ is an element $b\in B$, and such $b$ factorizes through the square above if and only if there exists a family $(a_n)_{n\in\omega}$ in $A$ (that is, a map $\omega\cdot 1\to A$) such that $(fa_n)_{n\in\omega}$ is a chain in $B$ and $\bigvee_nfa_n=b$. This is exactly expressing that $f$ is dense. See Section~\ref{omega-CPO} fo the corresponding notions of purity.

		 		\item Regular epimorphisms in a $\lambda$-quasivariety $\cv$ coincide with those maps $e$ for which $\cv_0(P,e)$ is surjective for every $\lambda$-presentable regular projective $P$ in $\cv$ (see Section~\ref{quasiv}). Thus, the class of regular epimorphisms in $\cv^\to$ can be characterized as that of those maps injective with respect to the family
		\begin{center}
			\begin{tikzpicture}[baseline=(current  bounding  box.south), scale=2]
				
				\node (a'0) at (0,0.8) {$0$};
				\node (b0) at (1.1,0.8) {$P$};
				\node (c0) at (0,0) {$P$};
				\node (d0) at (1.1,0) {$P$};
				
				\path[font=\scriptsize]

				(a'0) edge [->] node [above] {$!$} (b0)
				(b0) edge [->] node [right] {$d=1$} (d0)
				(c0) edge [->] node [below] {$y=1$} (d0)
				(a'0) edge [->] node [left] {$!$} (c0);
			\end{tikzpicture}	
		\end{center}
		for any $\lambda$-presentable and regular projective $P$.
	\end{enumerate}}
\end{exam}

\begin{lemma}\label{pointwise-E-pure}
	Let $\ck$ be a locally $\lambda$-presentable $\cv$-category. If Assumption~\ref{assumption} holds, the following conditions are equivalent for any $f\colon K\to L$ in $\ck$:\begin{enumerate}
		\item[(i)] $f$ is barely $(\lambda,\ce)$-pure;
		\item[(ii)] there exists a $\lambda$-filtered family $\cd\subseteq \ck^\to$ whose colimit is $f$ and such that $f$ is barely $\ce$-pure with respect to each morphism from $\cd$.
	\end{enumerate}
\end{lemma}
\begin{proof}
	$(i)\Rightarrow(ii)$. By definition of bare $(\lambda,\ce)$-purity and since $\ck$ is locally $\lambda$-presentable, it is enough to choose $\cd$ to be any $\lambda$-filtered family of $\lambda$-presentable objects of $\ck^\to$ with colimit $f$.
	
	$(ii)\Rightarrow(i)$. Fix a family $\cd=\{f_i\colon K_i\to L_i\}_{i\in I}$ as in condition $(ii)$ and consider any $g\colon A\to B$ between $\lambda$-presentable objects. We will use notation 
\ref{notation} with indices $i$'s denoting the $f_i$ to which they apply.
Then $\ck(B,K)\cong\colim_i\ck(B,K_i)$ and $\ck^\to(g,f)\cong\colim_i\ck^\to(g,f_i)$. Moreover, since $\cm$ and $\ce$ are closed under $\lambda$-filtered colimits in $\cv^\to$, also $\cq(g,f)\cong \colim_i\cq(g,f_i)$ and $q\cong\colim q_i$. We need to show that $q\colon\ck(B,K)\to\cq(g,f)$ lies in $\ce$.
	
	For that, write $\ce$ as an injectivity class $\ch$-inj in $\cv^\to$ as in the initial assumption. Thus, we are required to show that $q$ is injective with respect to each map $(!,y)\colon !_X\to d$ in $\ch$.
	
	Consider then a morphism $(!,v)\colon !_X\to q$, since $X$ is $\lambda$-presentable and $\cq(g,f)$ is the $\lambda$-filtered colimit of the $\cq(g,f_i)$ it follows that $v$ factors as a map $v_i\colon X\to \cq(g,f_i)$, for a fixed $i$. We can then construct the following commutative diagram:
	
	\begin{center}
		\begin{tikzpicture}[baseline=(current  bounding  box.south), scale=2]
			
			\node (21) at (-0.7,0) {$0$};
			\node (22) at (1.5,0) {$\ck(L_i,K)\otimes\ck^\to(g,f_i)$};
			\node (23) at (4,0) {$\ck(B,K)$};
			
			\node (31) at (-0.7,-0.9) {$X$};
			\node (32) at (1.5,-0.9) {$\cq(f_i,f)\otimes \cq(g,f_i)$};
			\node (33) at (4,-0.9) {$\cq(g,f)$};

			\path[font=\scriptsize]
			
			(21) edge [->] node [below] {$!$} (22)
			(21) edge [->] node [left] {$!_X$} (31)
			(22) edge [->>] node [right] {$q_{f_i}\otimes p'_i$} (32)
			(31) edge [->] node [above] {$v'$} (32)
			(22) edge [->] node [below] {$M$} (23)
			(23) edge [->] node [right] {$q$} (33)
			(32) edge [->] node [above] {$N$} (33)
			(21) edge [bend left=20,->] node [below] {$!$} (23)
			(31) edge [bend right=20,->] node [above] {$v$} (33);
		\end{tikzpicture}	
	\end{center} 
	where $M$ is induced by first applying the projection $\ck^\to(g,f_i)\to \ck(B,L_i)$ and then composing, and similarly $N$ is induced through the factorization system by the composition maps in $\ck$ and $\ck^\to$. Finally, the map $v'$ is given by tensoring the composite of the colimiting map $I\to \ck(f_i,f)$ and $p_i'\colon\ck(f_i,f)\to \cq(f_i,f)$ with $v_i$ (and composing with the isomorphism $X\cong I\otimes X$). 
	
	Now, $q_{f_i}\colon \ck(L_i,K)\to \cq(f_i,f)$ is in $\ce$ since $f$ is barely $\ce$-pure with respect to $f_i$, and 
$p_i'\colon \ck^\to(g,f_i)\to \cq(g,f_i)$ is in $\ce$ by definition. Since the factorization system is enriched, it follows that $q_{f_i}\otimes p'_i$ is also in $\ce$. Thus the pair $(!,v')$ factors through $(!,y)$; composing this with $(M,N)$ we obtain a factorization of $(!,v)$ through $(!,y)$. This shows that $q$ is injective with respect to $(!,y)$, and hence that $q\in\ce$.
\end{proof}

\begin{coro}\label{E-split-pure}
	If Assumption~\ref{assumption} holds, then every $\ce$-split morphism is barely $(\lambda,\ce)$-pure.
\end{coro}
\begin{proof}
	Let $s\colon K\to L$ be $\ce$-split; then $s$ is barely $\ce$-pure with respect to itself by Remark~\ref{relative}(2). Thus the family $\cd:=\{s\}$ satisfies condition (ii) above, making $s$ barely $(\lambda,\ce)$-pure.
\end{proof}

We are ready to prove a characterization theorem for $(\lambda,\ce)$-injectivity classes. This will be applied in the three very different settings of categories enriched over quantale-metric spaces (Section~\ref{quantales}), categories enriched over $\omega$-directed posets (Section~\ref{omega-CPO}), and categories enriched over $\lambda$-quasivarieties (Section~\ref{quasiv}). 

\begin{theo}\label{char-theo-1}
	Suppose that Assumption~\ref{assumption} holds and that
	\begin{enumerate}
		\item[(i)] barely $(\lambda,\ce)$-pure and $(\lambda,\ce)$-pure morphisms coincide;
		\item[(ii)] there is $\cg\subseteq\cv$ such that powers by $\cg$ are $\ce$-stable and whenever $\cv_0(G,e)$ is surjective for any $G\in\cg$, then $e\in\ce$.
	\end{enumerate}
	Then the $(\lambda,\ce)$-injectivity classes in a locally $\lambda$-presentable $\cv$-category $\ck$ are precisely the classes closed under $\lambda$-filtered colimits, products, powers by $\cg$, and $(\lambda,\ce)$-pure subobjects.
\end{theo}
\begin{proof}
Since $\ce$ is an injectivity class, it is closed in $\cv^\to$ under products. Thus the necessity follows from \ref{closure1}.

Conversely, let $\ca$ be closed under products, $\lambda$-filtered colimits, and $(\lambda,\ce)$-pure subobjects in $\ck$. Since $\ca$ is closed also under $\lambda$-pure subobjects by Corollary~\ref{pure-Epure}, its underlying category $\ca_0$ is a $\lambda$-injectivity class of $\ck_0$ by \cite[Theorem~2.2]{RAB}. Thanks to \cite[Theorem~3.3]{LT20} we then obtain a functor $R\colon\ck_0\to\ck_0$ preserving $\lambda$-filtered colimits and a natural transformation $r\colon 1_{\ck_0}\to R$ such that each component $r_K\colon K\to RK$ is an ordinary weak reflection of $K$ into $\ca_0$.

Now we move back to the enriched setting and consider the set $\ch$ of all those maps $g\colon A\to B$ between $\lambda$-presentable objects such that every object of $\ca$ is $\ce$-injective with respect to them. We shall prove that $\ca=\tx{Inj} \ch$.

The fact that $\ca\subseteq\tx{Inj} \ch$ is true by definition. Consider then $K\in\tx{Inj} \ch$ together with its ordinary weak reflection $r_K\colon K\to RK$ into $\ca$. To conclude it is enough to show that $r_K$ is $(\lambda,\ce)$-pure, so that $K\in\ca$ by the assumptions on $\ca$. For this, we use Lemma~\ref{pointwise-E-pure}.

Write $K\cong\colim K_i$ as a $\lambda$-filtered colimit of $\lambda$-presentable objects in $\ck$ with colimiting maps $x_i\colon K_i\to K$; since $R$ preserves $\lambda$-filtered colimits then $RK\cong\colim RK_i$ and $r_K\cong \colim\ r_{K_i}$ in the category of arrows $\ck^{\rightarrow}$. Now, by writing each $RK_i$ as a $\lambda$-filtered colimit
$$s_j:L^i_j\to RK_i$$
of $\lambda$-presentable objects, we get that each $r_{K_i}$ is a $\lambda$-filtered colimit of morphisms $h^i_j:K_i\to L^i_j$ in $K_i\downarrow\ck$ via certain $s^i_j:L^i_j\to RK_i$. It follows that the family $$\cd:=\{h^i_j:K_i\to L^i_j\}_{i,j}$$ is $\lambda$-directed, contained in $\ck_\lambda^{\rightarrow}$, and has colimit $r_K$. By Lemma~\ref{pointwise-E-pure}, to conclude it is enough to show that $r_K$ is $\ce$-pure with respect to each $h_j$.

Since $r_{K_i}=s_jh^i_j$ and $r_{K_i}$ is an ordinary weak reflection, every object from $\ca$ is injective in the ordinary sense with respect to $h^i_j$; that is $\ck_0(h^i_j,A)$ is surjective for any $A\in\ca$. Since by hypothesis $\ca$ is closed under powers by $\cg$, it follows that for any $A\in\ca$ and $G\in\cg$ the maps
$$ \cv_0(G,\ck(h^i_j,A))\cong \ck_0(h^i_j,G\pitchfork A) $$
are surjective. By $(ii)$ this means that $\ck(h^i_j,A)\in\ce$, and thus every object from $\ca$ is $\ce$-injective with respect to $h^i_j$. Therefore $K$ is $\ce$-injective with respect to $h^i_j$; that is, $\ck(h^i_j,K)$ is in $\ce$.
Now, following Notation~\ref{notation} for $g:=h^i_j$ and $f=r_K$ we have the following commutative triangle
\begin{center}
	\begin{tikzpicture}[baseline=(current  bounding  box.south), scale=2]
		
		\node (a0) at (-0.1,0.8) {$\ck(L^i_j,K)$};
		\node (c0) at (2.2,0.8) {$ \ck(K_i,K)$};
		\node (d0) at (1.1,0.2) {$\cq(h^i_j,r_K)$};
		
		\path[font=\scriptsize]
		
		(a0) edge [->>] node [above] {$\ck(h^i_j,K)$} (c0)
		(a0) edge [->] node [below] {$q\ \ \ $} (d0)
		(d0) edge [>->] node [below] {$\ \ \ p''$} (c0);
	\end{tikzpicture}
\end{center}
where $p''$ is in $\cm$. Since $\ck(h^i_j,K)$ is in $\ce$, also $p''$ is in $\ce$. Hence $p''$ is an isomorphism, and therefore $q$ is in $\ce$. Consequently, $r_K$ is $\ce$-pure with respect to each $h^i_j$, as claimed.
\end{proof}

\begin{exam}
Let $\cv=\Ban$ with $(\ce,\cm)=$ (dense, isometry); then: $\ce$ satisfies Assumptions~\ref{assumption}, the two notions of purity coincide (by Example~\ref{Ban2}), and $\cg:=\{\mathbb C\}$ satisfies condition (ii) above.
	
	Therefore we can apply Theorem~\ref{char-theo-1} and obtain the following characterization: the $(\lambda,\ce)$-injectivity classes in a locally $\lambda$-presentable $\Ban$-category $\ck$ are precisely the classes closed under $\lambda$-filtered colimits, products, and $(\lambda,\ce)$-pure subobjects.
\end{exam}

In the same spirit as $\ce$-injectivity we can consider orthogonality: given a morphism $f\colon A\to B$ in a $\cv$-category $\ck$, we say that $X\in\ck$ is {\em $f$-orthogonal} if the map
$$ \ck(f,X)\colon\ck(B,X)\longrightarrow\ck(A,X) $$
is an isomorphism. This is the same as considering $\ce$-injectivity with respect to $\ce=\cv^{\cong}$; note, however, that we shall not fix such choice of $\ce$ when considering $\ce$-purity below.

Given a set of morphisms $\ch$ in $\ck$, we denote by $\ch^\perp$ the full subcategory of $\ck$ spanned by the object which are $f$-orthogonal for any $f\in\ch$. We call {\em orthogonality class} any full subcategory of $\ck$ arising that way; if the morphisms in $\ch$ have $\lambda$-presentable domain and codomain, we call $\ch^\perp$ a {\em $\lambda$-orthogonality class}. See \cite[Chapter~6]{Kel82:book} where the enriched concept was first considered. 

\begin{propo}\label{closure-orth}
	Every $\lambda$-orthogonality class $\ch^\perp$ in a complete and cocomplete $\cv$-category $\ck$ is closed under:\begin{enumerate}
		\item all limits;
		\item $\lambda$-filtered colimits.		
	\end{enumerate}
	If moreover, $\ce_\lambda^\perp=\cm$ and $[e,X]\in\cm$ whenever $e\in\ce$ and $X\in\cv$, then $\ch^\perp$ is also closed under
	\begin{enumerate}
		\item[(3)] $(\lambda,\ce)$-pure subobjects.
	\end{enumerate}
\end{propo}
\begin{proof}
	Points (1) and (2) follows from Propositions 2.6 and 2.7 of \cite{LR12}.
	
	For (3), assume that $f\colon K\to L$ is $(\lambda,\ce)$-pure and that $L\in\ch^\perp$. For any $g\colon A\to B$ in $\ch$, following Notation~\ref{strong-pure}, we obtain that $\cp(g,L)\cong \ck(A,L)$ (since $\ck(g,L)$ is an isomorphism), and hence that $\cp(g,f)\cong\ck(A,K)$. Thus, the $\ce$-purity of $f$ with respect to $g$ implies that $\ck(g,K)$ lies in $\ce$. But by Proposition~\ref{good1} we know that $\ck(B,f)$ and $\ck(A,f)$ are in $\cm$, so that $\ck(g,K)$ is in $\cm$ too. It follows that $\ck(g,K)$ is an isomorphism and hence that $K\in\ch^\perp$.
\end{proof}

As a consequence of the main theorem of this section we obtain:

\begin{coro}\label{orth->inj}
	Assume that $\ce_\lambda^\perp=\cm$ and $[e,X]\in\cm$ whenever $e\in\ce$ and $X\in\cv$. If the hypotheses of Theorem~\ref{char-theo-1} are satisfied, then every $\lambda$-orthogonality class in a locally $\lambda$-presentable $\cv$-category $\ck$ is a $(\lambda,\ce)$-injectivity class.
\end{coro}
\begin{proof}
	Follows immediately from Proposition~\ref{closure-orth} above and Theorem~\ref{char-theo-1}.
\end{proof}

\begin{rem}
	Given a set of maps $\ch$ in $\ck$, ordinarily there is an explicit way of constructing (using pushouts) another collection $\ch'$ for which
	$$ \ch^\perp=\tx{Inj} \ch',$$
	see \cite[Remark~4.4.1]{AR}. It is not clear to us if such an argument would work in the enriched setting for a given factorization system $(\ce,\cm)$ satisfying the hypotheses of the corollary above.
\end{rem}

\section{Purity over quantale-valued metric spaces}\label{quantales}
Under a \textit{quantale} $Q$ we will mean a complete lattice $(Q,\leq)$ with a commutative monoid structure $(Q,+,0)$ where $0$ is the least element of $Q$ such that 
$$
x+\bigwedge\limits_{j\in J}y_j=\bigwedge\limits_{j\in J}(x+ y_j)
$$  
for every $x,y_j\in Q$, $j\in J$. Given $a,b\in Q$ we say that \textit{$a$ is totally above $b$}, and write $a\succ b$, if: for every subset $S\subseteq Q$ such that $a\geq\bigwedge S$ there exists $s\in S$ such that $b\geq s$. Let now
$\Uparrow(a)=\{x\in Q|x\succ a\}$; a \textit{value quantale} is a quantale $Q$ such that $$a=\bigwedge\Uparrow(a)$$ for every $a\in Q$ and,
moreover, $a\wedge b\succ 0$ whenever $a,b\succ 0$ (see \cite{F,CW}).
  
\begin{defi}\label{Q-met}
Let $Q$ be a value quantale.
A $Q$-\textit{metric space} is a set $X$ equipped with a map $d:X\times X\to Q$ satisfying the following conditions for all $x,y,z\in X$:
\begin{enumerate} 
\item $d(x,x)=0$,
\item $d(x,y)=d(y,x)$,
\item $d(x,z)\leq d(x,y)+ d(y,z)$, and
\item $d(x,y)=0\Rightarrow x=y$.
 \end{enumerate}
A map $f:X\to Y$ is \textit{non-expanding} if $d(fx,fy)\leq d(x,y)$ for all $x,y\in X$. The resulting category is denoted by $Q$-$\Met$.
\end{defi}

\begin{exams}
(1) Metric spaces are $Q$-valued metric spaces for $Q=([0,\infty],\leq,+)$.

(2) Ultrametric spaces are $Q$-quantales for $Q=((0,\infty],\leq,\vee)$.

(3) Probabilistic metric spaces are $Q$-metric spaces where $Q$ is the value quantale  of distance distribution functions. These are left continuous maps $f:[0,\infty]\to[0,1]$, i.e., maps such that 
$$
f(x)=\bigvee\limits_{y<x}f(y).
$$ 
In particular, $f(0)=0$. The order on $Q$ is the opposite of the pointwise order and the operation
is
$$
(f+g)(x)=\bigvee\limits_{y+z\leq x}\max\{f(y)+ g(z)-1,0\}.
$$
(see \cite{F}, or \cite[1.2(4)]{CH}).
\end{exams}

The \textit{tensor product} $X\otimes Y$ of $Q$-valued metric spaces is given by the set
$X\times Y$ with the distance
$$
d((x_1,y_1),(x_2,y_2))=d(x_1,x_2)+ d(y_1,y_2).
$$
The tensor unit $I$ is the one-element $Q$-metric space.
The internal hom $[X,Y]$ is the $Q$-metric space on the set $Q$-$\Met(X,Y)$ with distance
$$
d(f,g)=\bigvee\limits_{x\in X}d(fx,gx).
$$

	

\begin{propo}
$Q$-$\Met$ is a symmetric monoidal closed category.
\end{propo}
\begin{proof}
Clearly $\otimes$ is symmetric and $I$ is a unit for it. Consider a nonexpanding map $f:X\otimes Y\to Z$. Since
$$
d(f(x,y_1),f(x,y_2))\leq d((x,y_1),(x,y_2))=d(x,x)+ d(y_1,y_2)=
d(y_1,y_2),
$$
the maps $f(x,-)\colon Y\to Y$ are nonexpanding. Moreover, since
\begin{align*}
d(f(x_1,-),f(x_2,-))&=\bigvee\limits_{y\in Y}d(f(x_1,y),f(x_2,y))\leq
\bigvee\limits_{y\in Y}d((x_1,y),(x_2,y))\\
&=\bigvee\limits_{y\in Y}(d(x_1,x_2)+ d(y,y))= d(x_1,x_2),
\end{align*}
the map $\tilde{f}\colon X\to[Y,Z]$ such that $\tilde{f}x =f(x,-)$ is nonexpanding.

Conversely, consider a nonexpanding map $g:X\to[Y,Z]$ and take
$\tilde{g}:X\otimes Y\to Z$ such that $\tilde{g}(x,y)=g(x)(y)$. Since
\begin{align*}
d(\tilde{g}(x_1,y_1),\tilde{g}(x_2,y_2))&=d(g(x_1)(y_1),g(x_2)(y_2))\\
&\leq d(g(x_1)(y_1),g(x_1)(y_2))+ d(g(x_1)(y_2),g(x_2)(y_2))\\
&\leq d(y_1,y_2) +\bigvee\limits_{y\in Y}d(g(x_1)(y),g(x_2)(y))\\
&=d(y_1,y_2) +d(g(x_1),g(x_2))\\
&\leq d(y_1,y_2)+d(x_1,x_2)=d((x_1,y_1),(x_2,y_2)),
\end{align*}
the map $\tilde{g}$ is nonexpanding. This defines a natural bijection between maps $X\otimes Y\to Z$ and $X\to[Y,Z]$.
\end{proof}

\begin{propo}
$Q$-$\Met$ is a locally presentable category.
\end{propo}
\begin{proof}
$Q$-metric spaces can be seen as sets equipped with binary relations
$R_q$, $q\in Q$ such that
\begin{enumerate}
\item $(\forall x,y)(R_0(x,y)\leftrightarrow x=y)$,
\item $(\forall x,y)(R_q(x,y)\rightarrow R_q(y,x))$ for all $q\in Q$, 
\item  $(\forall x,y)(R_p(x,y)\rightarrow R_q(x,y))$ for all $p\leq q$
\item $(\forall x,y,z)(R_p(x,z)\wedge R_q(z,y)\rightarrow R_{p+q}(x,y))$, for all $p,q\in Q$, and
\item $(\forall x,y) (\bigwedge\limits_{j\in J} R_{q_j}(x,y)\rightarrow R_q(x,y))$ where $q=\bigwedge\limits_{j\in J}q_j$.
\end{enumerate}
The result follows from \cite[5.30]{AR}.
\end{proof}

\begin{rem}\label{lp}
Following \cite[5.30]{AR}, $Q$-$\Met$ is locally $\lambda$-presentable whenever $\lambda>|Q|$. It even suffices that $\lambda>|S|$ where $S$ is a set of elements $s>0$ such that $0=\bigwedge S$. Then $q=\bigwedge\limits_{s\in S} q+s$ and we only use the corresponding conjunctions in (5).

Directed colimits of $Q$-metric spaces are calculated in the same way
as in metric spaces (see \cite[2.4]{AR1}): given a directed diagram $\{K_i\}_i$, the cocone $k_i:K_i\to K$ is colimiting if and only if it is jointly surjective and 
$$
d(k_i(x),k_i(y))=\bigwedge\limits_{j\geq i}d(k_{ij}(x),k_{ij}(y)).
$$
\end{rem}

A map $f:X\to Y$ is an \textit{isometry} if $d(fx,fy)=d(x,y)$ for all $x,y\in X$. Following \ref{lp}, isometries are closed under directed colimits in ($Q$-$\Met)^\to$.

\begin{rem}
	A key property of value quantales is that for every $q\succ 0$ there is 
	$p\succ 0$ such that $q\succ 2p$ (see \cite[2.9]{F}). This makes possible
	to define a topology on a $Q$-metric space $(X,d)$ as follows (see \cite[4.1]{F}). For $x\in X$ and $q\succ 0$, the \textit{open ball} $B_q(x)$ is the set $\{y\in X|q\succ d(x,y)\}$. Now, a set $U\subseteq X$ is \textit{open} if for every $x\in U$ there is $q\succ 0$ such that $B_q(x)\subseteq U$.
\end{rem}

A nonexpanding map $f:X\to Y$ will be called \textit{dense} if for every
$y\in Y$ and every $q\succ 0$ there is $x\in X$ such that $q\succ d(fx,y)$.
These are precisely maps dense in the induced topologies. We have 
an enriched and proper factorization system (dense, closed isometry) in
$Q$-$\Met$. Moreover, the class $\ce$ of dense maps is an injectivity
class in $(Q$-$\Met)^\to$ with respect to maps
\begin{center}
			\begin{tikzpicture}[baseline=(current  bounding  box.south), scale=2]
				
				\node (a'0) at (0,0.8) {$0$};
				\node (b0) at (1.1,0.8) {$1$};
				\node (c0) at (0,0) {$1$};
				\node (d0) at (1.1,0) {$2_q$};
				
				\path[font=\scriptsize]

				(a'0) edge [->] node [above] {$!$} (b0)
				(b0) edge [->] node [right] {$d=i_1$} (d0)
				(c0) edge [->] node [below] {$y=i_0$} (d0)
				(a'0) edge [->] node [left] {$!$} (c0);
			\end{tikzpicture}	
		\end{center}
		for every $q\succ 0$. Here $2_q$ is the $Q$-metric space with two points of distance $q$ (see \ref{ex-ass}(2)).

\begin{rem}\label{sat-assmpt}
	It follows that $Q$-$\Met$ satisfies Assumption~\ref{assumption} for the regular cardinal $\lambda$ given by Remark~\ref{lp}, for $\ce$ = dense and $\cm$ = closed isometry. Moreover $\ce$ contains the surjections.
\end{rem}

Given a $Q$-$\Met$ enriched category $\ck$ and morphisms $f,f':K\to L$, we write $f\sim_q f'$ for $q\succ 0$ if $d(f,f')\leq q$. A $q$-commutative square
\begin{center}
	\begin{tikzpicture}[baseline=(current  bounding  box.south), scale=2]
		
		\node (a'0) at (0,0.8) {$A$};
		\node (b0) at (0.9,0.8) {$B$};
		\node (c0) at (0,0) {$K$};
		\node (d0) at (0.9,0) {$L$};
		
		\path[font=\scriptsize]

		(a'0) edge [->] node [above] {$g$} (b0)
		(b0) edge [->] node [right] {$v$} (d0)
		(c0) edge [->] node [below] {$f$} (d0)
		(a'0) edge [->] node [left] {$u$} (c0);
	\end{tikzpicture}	
\end{center}
is a square such that $fu\sim_q vg$. In this setting the notion of $q$-pushout that we introduce below is relevant: this will play a key role in the proof of Proposition~\ref{equ-purities}. Like in \cite{AR1}, $q$-pushouts can be seen as weighted ($\aleph_1$-small) colimits.

\begin{defi}
	Let $q\succ 0$. A $q$-commutative square
	\begin{center}
		\begin{tikzpicture}[baseline=(current  bounding  box.south), scale=2]
			
			\node (a'0) at (0,0.8) {$A$};
			\node (b0) at (0.9,0.8) {$B$};
			\node (c0) at (0,0) {$C$};
			\node (d0) at (0.9,0) {$D$};
			
			\path[font=\scriptsize]

			(a'0) edge [->] node [above] {$f$} (b0)
			(b0) edge [->] node [right] {$\bar g$} (d0)
			(c0) edge [->] node [below] {$\bar f$} (d0)
			(a'0) edge [->] node [left] {$g$} (c0);
		\end{tikzpicture}	
	\end{center}
	is called an $q$-\textit{pushout} if for every $q$-commutative square
	\begin{center}
		\begin{tikzpicture}[baseline=(current  bounding  box.south), scale=2]
			
			\node (a'0) at (0,0.8) {$A$};
			\node (b0) at (0.9,0.8) {$B$};
			\node (c0) at (0,0) {$C$};
			\node (d0) at (0.9,0) {$D'$};
			
			\path[font=\scriptsize]

			(a'0) edge [->] node [above] {$f$} (b0)
			(b0) edge [->] node [right] {$g'$} (d0)
			(c0) edge [->] node [below] {$f'$} (d0)
			(a'0) edge [->] node [left] {$g$} (c0);
		\end{tikzpicture}	
	\end{center}
	there is a unique morphism $t:D\to D'$ such that $t\overline{f}=f'$ and $t\overline{g}=g'$.
\end{defi}

Let us now consider the notions of purity induced in this setting. We start with the barely $(\lambda,\ce)$-pure morphisms of Definition~\ref{barely-pure-def}. Given $f\colon K\to L$ in a $Q$-$\Met$ category $\ck$ and $g\colon A\to B$ in $\ck_\lambda$, the object $\cq(g,f)$ of Notation~\ref{notation} is the sub-$Q$-metric space of $\ck(A,K)$ consisting of those $u:A\to K$ such that for every $q\succ 0$ there are $u':A\to K$ and $v:B\to L$ such that $u\sim_{q}u'$ and $fu'=vg$. Arguing as in Example~\ref{ex2} we obtain that $f:K\to L$ is {\em barely $(\lambda,\ce)$-pure} if for every $q\succ 0$ and every commutative square 
\begin{center}
	\begin{tikzpicture}[baseline=(current  bounding  box.south), scale=2]
		
		\node (a'0) at (0,0.8) {$A$};
		\node (b0) at (0.9,0.8) {$B$};
		\node (c0) at (0,0) {$K$};
		\node (d0) at (0.9,0) {$L$};
		
		\path[font=\scriptsize]

		(a'0) edge [->] node [above] {$g$} (b0)
		(b0) edge [->] node [right] {$v$} (d0)
		(c0) edge [->] node [below] {$f$} (d0)
		(a'0) edge [->] node [left] {$u$} (c0);
	\end{tikzpicture}	
\end{center}
with $A$ and $B$ $\lambda$-presentable, there exists $t:B\to K$ such that $tg\sim_q u$. Therefore this generalizes the notion of barely $\lambda$-ap-pure morphism of \cite{RT}.

The $(\lambda,\ce)$-pure morphisms of Definition~\ref{pure-def1} do not correspond directly to a notion of purity studied in the past, as pointed out in Example~\ref{ex1}(2) for $\cv=\CMet$. Given $f\colon K\to L$ and $g\colon A \to B$ in a $Q$-$\Met$ category $\ck$, the object $\cp(g,f)$ of Notation~\ref{strong-pure} is the sub-$Q$-metric space of $\ck(A,K)$ consisting of those $u:A\to K$ such that for every $q\succ 0$ there exists $v:B\to L$ for which $fu\sim_{q}vg$. Then $f$ is {\em $\ce$-pure with respect to $g$} if and only if for any $u\in \cp(f,g)$ as above and for any $q\succ 0$, there exists $t\colon B\to K$ such that $tg\sim_{q}u$.

A third notion that we could consider is a direct generalization to this setting of the weakly $\lambda$-ap pure morphisms of \cite{RT}. We say that a morphism $f:K\to L$ 
in a $Q$-$\Met$ category $\ck$ is {\em weakly $(\lambda,\ce)$-pure} if and only if for every $q\succ 0$
and every $q$-commutative square
\begin{center}
	\begin{tikzpicture}[baseline=(current  bounding  box.south), scale=2]
		
		\node (a'0) at (0,0.8) {$A$};
		\node (b0) at (0.9,0.8) {$B$};
		\node (c0) at (0,0) {$K$};
		\node (d0) at (0.9,0) {$L$};
		
		\path[font=\scriptsize]

		(a'0) edge [->] node [above] {$g$} (b0)
		(b0) edge [->] node [right] {$v$} (d0)
		(c0) edge [->] node [below] {$f$} (d0)
		(a'0) edge [->] node [left] {$u$} (c0);
	\end{tikzpicture}	
\end{center}
with $A$ and $B$ $\lambda$-presentable, there exists $t:B\to K$ such that $tg\sim_{2q} u$. 		

These three notions are actually all equivalent, as we see in the result below that extends \cite[4.2]{R}.

\begin{propo}\label{equ-purities}
Let $\ck$ be a locally $\lambda$-presentable category enriched over $Q$-$\Met$ and $f\colon K\to L$ be a morphism in $\ck$. The following are equivalent:\begin{enumerate}
	\item $f$ is weakly $(\lambda,\ce)$-pure;
	\item $f$ is $(\lambda,\ce)$-pure;
	\item $f$ is barely $(\lambda,\ce)$-pure.
\end{enumerate}
\end{propo}	
\begin{proof}
	$(1)\Rightarrow (2)$ is straightforward since every $u\in\cp(g,f)$ can be completed to a $q$-commutative square as in the definition of weak $(\lambda,\ce)$-purity, while $(2)\Rightarrow (3)$ is given by Lemma~\ref{pure-bpure}.
	
$(3)\Rightarrow (1)$. Let $f$ be barely $(\lambda,\ce)$-pure and consider a 
$q$-commutative square
\begin{center}
	\begin{tikzpicture}[baseline=(current  bounding  box.south), scale=2]
		
		\node (a'0) at (0,0.8) {$A$};
		\node (b0) at (0.9,0.8) {$B$};
		\node (c0) at (0,0) {$K$};
		\node (d0) at (0.9,0) {$L$};
		
		\path[font=\scriptsize]

		(a'0) edge [->] node [above] {$g$} (b0)
		(b0) edge [->] node [right] {$v$} (d0)
		(c0) edge [->] node [below] {$f$} (d0)
		(a'0) edge [->] node [left] {$u$} (c0);
	\end{tikzpicture}	
\end{center}
with $A$ and $B$ $\lambda$-presentable. Consider a $q$-pushout as below;
\begin{center}
	\begin{tikzpicture}[baseline=(current  bounding  box.south), scale=2]
		
		\node (a'0) at (0,0.8) {$A$};
		\node (b0) at (0.9,0.8) {$B$};
		\node (c0) at (0,0) {$A$};
		\node (d0) at (0.9,0) {$C$};
		
		\path[font=\scriptsize]

		(a'0) edge [->] node [above] {$g$} (b0)
		(b0) edge [->] node [right] {$g_q$} (d0)
		(c0) edge [->] node [below] {$\overline g$} (d0)
		(a'0) edge [->] node [left] {$\id_A$} (c0);
	\end{tikzpicture}	
\end{center}
then there is a unique morphism $t:C\to L$ such that $t\overline{g}=fu$ and
$tg_q=v$. Thus we get the following commutative square
\begin{center}
	\begin{tikzpicture}[baseline=(current  bounding  box.south), scale=2]
		
		\node (a'0) at (0,0.8) {$A$};
		\node (b0) at (0.9,0.8) {$C$};
		\node (c0) at (0,0) {$K$};
		\node (d0) at (0.9,0) {$L$};
		
		\path[font=\scriptsize]

		(a'0) edge [->] node [above] {$\overline g$} (b0)
		(b0) edge [->] node [right] {$t$} (d0)
		(c0) edge [->] node [below] {$f$} (d0)
		(a'0) edge [->] node [left] {$u$} (c0);
	\end{tikzpicture}	
\end{center}
and, since $\lambda$-presentable objects are closed under weighted finite colimits, $C$ is $\lambda$-presentable and thus there exists $w:C\to K$ such that $w\overline{g}\sim_q u$. Hence
$$
wg_q g\sim_q w\overline{g}\sim_q u
$$
and $wg_q g\sim_{2q } u$, showing that $f$ is weakly $(\lambda,\ce)$-pure.  
\end{proof}

As a consequence we can characterize $(\lambda,\ce)$-injectivity classes:

\begin{theo}
	The $(\lambda,\ce)$-injectivity classes in a locally $\lambda$-presentable $Q$-$\Met$ category $\ck$ are precisely classes closed under $\lambda$-filtered colimits, products, and $(\lambda,\ce)$-pure subobjects.
\end{theo}
\begin{proof}
	This is a consequence of Theorem~\ref{char-theo-1} since Assumption~\ref{assumption} is satisfied by Remark~\ref{sat-assmpt}, the notions of purity coincide by Proposition~\ref{equ-purities} above, and we can choose $\cg:=\{1\}$ since surjections are dense.
\end{proof}

\section{Purity over $\omega$-complete posets}\label{omega-CPO}

We fix as the base of enrichment the cartesian closed category  $\omega$-$\CPO$  of posets with joins of non-empty $\omega$-chains and maps preserving joins of such (as in Example~\ref{ex-ass}(4)). This is locally $\aleph_1$-presentable as a closed category where $\aleph_1$-presentable objects are countable cpo's, so we consider $\lambda\geq\aleph_1$.

The factorization system we consider is given by (dense, embedding) as in Example~\ref{ex-ass}(4). Note that this is a proper factorization system, it is enriched (if $e$ is dense, then so is also $X\times e$ for any $X\in\omega$-$\CPO$), and the class of embeddings is closed under $\aleph_1$-filtered colimits in $\omega$-$\CPO^\to$. Thus Assumption~\ref{assumption} is satisfied.

Following Notation~\ref{strong-pure} for $f\colon K\to L$ and $g\colon A\to B$ the object $\cp(f,g)$ consists of those $u\colon A\to K$ for which there exists $(v_i\colon B\to L)_{i\in\omega}$ such that $(v_ig)_{i\in\omega}$ is a chain and $fu=\bigvee v_ig$. (Note that this does not mean that $fu=vg$, for some $v$, because $(v_i)_{i\in\omega}$ does not need to be a chain.) It follows that $f\colon K\to L$ is {\em $\ce$-pure} with respect to $g\colon A\to B$ if:
\vspace{5pt} 
\begin{center}\em
	 for any $u\colon A\to K$ together with a family $(v_i\colon B\to L)_{i\in\omega}$\\ such that $(v_ig)_{i\in\omega}$ is a chain and $fu=\bigvee v_ig$,\\ then\\ there exists $(t_j\colon B\to K)_{j\in\omega}$ such that $(t_jg)_j$ is a chain and $u=\bigvee t_jg$.
\end{center}
\vspace{5pt}
Similarly, following Notation~\ref{notation} for $f\colon K\to L$ and $g\colon A\to B$ as above, the object $\cq(f,g)$ consists of those $u\colon A\to K$ for which there exists $(u_i\colon A\to K,v_i\colon B\to L)_{i\in\omega}$ such that $(u_i)_{i\in\omega}$ is a chain, $fu_i=v_ig$, and $u= \bigvee u_i$. It follows that $f\colon K\to L$ is {\em barely $\ce$-pure} with respect to $g\colon A\to B$ if: 
\vspace{5pt}
\begin{center}\em
	for any $u\colon A\to K$ together with a family $(u_i\colon A\to K,v_i\colon B\to L)_{i\in\omega}$,\\ such that $(u_i)_{i\in\omega}$ is a chain, $fu_i=v_ig$, and $u= \bigvee u_i$,\\ then\\ there exists $(t_j\colon B\to K)_{j\in\omega}$ such that $(t_jg)_j$ is a chain and $u=\bigvee t_jg$.
\end{center}
\vspace{5pt}

As in the previous section, we shall prove that $(\lambda,\ce)$-pure and barely $(\lambda,\ce)$-pure morphisms coincide by showing that that correspond to a third (weaker) notion of purity. Before achieving that, we need to introduce a generalized notion of pushout:

\begin{defi}
	The {\em $\omega$-pushout} of a map $g\colon A\to B$ along $f\colon A\to C$ in a $\omega$-$\CPO$-category $\ck$ is an object $D$ together with maps $\bar g\colon C\to D$ and $\bar f_i\colon B\to D$, for $i\in \omega$, such that: \begin{itemize}
		\item the composites $(\bar f_i  g)_{i\in\omega}$ define a chain in $\ck(A,D)$,
		\item $\bar g f=\bigvee \bar f_i g$,
	\end{itemize}
	and which is universal among the tuples $(E,h\colon C\to E,(k_i\colon B\to E)_{i\in \omega})$ with the above properties: for any such tuple there exists a unique map $v\colon D\to E$ with $h= v\bar g$ and $k_i= v\bar f_i$ for any $i\in\omega$.
\end{defi}

\begin{rem}
	The $\omega$-pushout defined above can be compute as an $\aleph_1$-small weighted colimit in $\ck$. Indeed, consider the free $\omega$-$\CPO$-category $\cc$
	on the category 
	\begin{center}
		\begin{tikzpicture}[baseline=(current  bounding  box.south), scale=2]
			
			\node (a'0) at (0,0.8) {$\star_1$};
			\node (b0) at (0.9,0.8) {$\star_2$};
			\node (c0) at (0,0) {$\star_3$};
			
			\path[font=\scriptsize]

			(a'0) edge [->] node [above] {} (b0)
			(a'0) edge [->] node [left] {} (c0);
		\end{tikzpicture}	
	\end{center}
	and the weight $M\colon\cc^{\op}\to\omega$-$\CPO$ defined by sending the opposite of the diagram above to
	\begin{center}
		\begin{tikzpicture}[baseline=(current  bounding  box.south), scale=2]

			\node (b0) at (1,0.8) {$\omega\cdot 1$};
			\node (c0) at (0,0) {$1$};
			\node (d0) at (1,0) {$\omega+1$};
			
			\path[font=\scriptsize]

			(b0) edge [->] node [right] {$d$} (d0)
			(c0) edge [->] node [below] {$y$} (d0);
		\end{tikzpicture}	
	\end{center}
	with the same notations we used in Example~\ref{ex-ass}(4) to define the dense maps as an injectivity class.
	
	Then, to give $g\colon A\to B$ and $f\colon A\to C$ in $\ck$ is the same as giving an enriched functor $H\colon\cc\to\ck$, and it is easy to see that the $\omega$-pushout of $g$ along $f$ coincides with the weighted colimit $M*H$.
\end{rem}

We are now ready to prove the two notions of purity coincide (like
in \ref{equ-purities}):

\begin{propo}\label{equ-purities-CPO}
	Let $\ck$ be a locally $\lambda$-presentable category enriched over $\omega$-$\CPO$ and $f\colon K\to L$ be a morphism in $\ck$. The following are equivalent:\begin{enumerate}
		\item $f$ is $(\lambda,\ce)$-pure;
		\item $f$ is barely $(\lambda,\ce)$-pure;
		\item for any $g\colon A\to B$ between $\lambda$-presentable objects and any commutative square
		\begin{center}
			\begin{tikzpicture}[baseline=(current  bounding  box.south), scale=2]
				
				\node (a'0) at (0,0.8) {$A$};
				\node (b0) at (0.9,0.8) {$B$};
				\node (c0) at (0,0) {$K$};
				\node (d0) at (0.9,0) {$L$};
				
				\path[font=\scriptsize]

				(a'0) edge [->] node [above] {$g$} (b0)
				(b0) edge [->] node [right] {$v$} (d0)
				(c0) edge [->] node [below] {$f$} (d0)
				(a'0) edge [->] node [left] {$u$} (c0);
			\end{tikzpicture}	
		\end{center}
		There exists $(t_i\colon B\to K)_{i\in\omega}$ such that $(t_ig)_{i\in\omega}$ is a chain and $u=\bigvee t_ig$.
	\end{enumerate}
\end{propo}	
\begin{proof}
	$(1)\Rightarrow(2)$ follows from Lemma~\ref{pure-bpure}, and $(2)\Rightarrow(3)$ is trivial since it is enough to apply the bare $(\lambda,\ce)$-purity of $f$ to the constant family $(u_i,v_i)=(u,v)$.
	
	$(3)\Rightarrow(1)$. Following the definition of $(\lambda,\ce)$-purity, we need to consider $g:A\to B$ between $\lambda$-presentable objects, a map $u:A\to K$, and a family $v_i\colon B\to L$ such that $v_ig$ is a chain and $fu=\bigvee v_ig$.
	
	Consider now the $\omega$-pushout of $g$ along $1_A$ in $\ck$; this is an object $D$ together with maps $\bar g\colon A\to D$ and $ h_i\colon B\to D$, for $i\in \omega$, such that $(h_i  g)_{i\in\omega}$ is a chain and $\bar g =\bigvee h_i g$. By the universal property of the $\omega$-pushout applied to the tuple $(L,fu,(v_i)_{i\in\omega})$ there exists a map $\bar v\colon D\to L$ making the square
	\begin{center}
		\begin{tikzpicture}[baseline=(current  bounding  box.south), scale=2]
			
			\node (a'0) at (0,0.8) {$A$};
			\node (b0) at (0.9,0.8) {$D$};
			\node (c0) at (0,0) {$K$};
			\node (d0) at (0.9,0) {$L$};
			
			\path[font=\scriptsize]

			(a'0) edge [->] node [above] {$\bar g$} (b0)
			(b0) edge [->] node [right] {$\bar v$} (d0)
			(c0) edge [->] node [below] {$f$} (d0)
			(a'0) edge [->] node [left] {$u$} (c0);
		\end{tikzpicture}	
	\end{center}
	commute. Now, by applying (3) to $\bar g$ (which is still a map between $\lambda$-presentable objects because $D$ is an $\aleph_1$-small weighted colimit of $\lambda$-presentable objects) we find a family $t_i\colon D\to K$, for $i\in\omega$, such that $(t_i\bar g)_i$ is a chain and $u=\bigvee t_i\bar g$. It follows that
	$$ u=\bigvee_i t_i\bar g = \bigvee_i t_i(\bigvee_j h_j g)= \bigvee_i \bigvee_j  (t_i h_j) g = \bigvee_i (t_ih_i) g$$
	where the last equality is given by the standard diagonal argument for joins of chains in posets. In conclusion, the family $(t_ih_i)_{i\in\omega}$ witnesses the $(\lambda,\ce)$-purity of $f$, showing (1).
\end{proof}

As a consequence we can characterize $(\lambda,\ce)$-injectivity classes for $\omega$-$\CPO$-enriched categories:

\begin{theo}
	The $(\lambda,\ce)$-injectivity classes in a locally $\lambda$-presentable $\omega$-$\CPO$ category $\ck$ are precisely classes closed under $\lambda$-filtered colimits, products, and $(\lambda,\ce)$-pure subobjects.
\end{theo}
\begin{proof}
	This is a consequence of Theorem~\ref{char-theo-1} since Assumption~\ref{assumption} is satisfied, the notions of purity coincide by Proposition~\ref{equ-purities-CPO} above, and we can choose $\cg:=\{1\}$ since surjections are dense.
\end{proof}

\section{Purity over quasivarieties}\label{quasiv}

In this section we consider enrichment over $\lambda$-quasivarieties as in \cite{LT20}; we recall the definition of symmetric monoidal $\lambda$-quasivariety below:

\begin{defi}
	Let $\cv=(\cv_0,\otimes,I)$ be a symmetric monoidal closed category. We say that $\cv$ is a {\em symmetric monoidal $\lambda$-quasivariety} if:\begin{enumerate}
		\item $\cv_0$ has a regular generator $\cp$ made of $\lambda$-presentable and regular projective objects (that is, $\cv_0$ is a $\lambda$-quasivariety);
		\item the unit $I$ is $\lambda$-presentable;
		\item if $P,Q\in\cp$ then $P\otimes Q$ is $\lambda$-presentable and regular projective.
	\end{enumerate}
\end{defi}

In particular, such a $\cv$ is also locally $\lambda$-presentable as a closed category in the sense of Kelly \cite{Kel82} (see \cite[Remark~4.15]{LT20}). Following \cite[Proposition 30]{AR2}, a regular generator as in (1) can be weakened to be just a strong generator.

Since every $\lambda$-quasivariety is a regular category, it is endowed with the factorization system $(\ce,\cm)$ where $\ce$ is the class of regular epimorphisms and $\cm$ that of monomorphisms. We consider then the notion of purity induced by this factorization system. 

\begin{exams}
	The following are all examples of symmetric monoidal finitary quasivarieties, some were already mentioned in Examples~\ref{ex1}:\begin{enumerate}
		\item $\bo{Set}$, where the induced factorization system induces the standard notion of purity;
		\item $\MGra$, the cartesian closed category of oriented multigraphs with loops;
		\item $\bo{Gra}$, the cartesian closed category of directed graphs (possibly) without loops;
		\item $\bo{Set}^G$, the cartesian closed category of $G$-sets for a finite group $G$;
		\item $\bo{Ab}$, the category of abelian groups, as well as the category $R\tx{-}\bo{Mod}$ of $R$-modules, for a commutative ring $R$;
		\item $\bo{GAb}$, the category of graded abelian groups;
		\item $\bo{DGAb}$, the category of differentially graded abelian groups.
	\end{enumerate}
	In (1),(2),(5), and (6) the unit of the monoidal structure is regular projective, so the notion of $\ce$-purity will correspond to the ordinary one by Proposition~\ref{pure-E-pure} below. This does not hold in the other examples. We will keep $\bo{DGAb}$ as our prototypical example of quasivariety with non regular projective unit, and will provide clarifications on how the various notions introduced below can be interpreted in the setting of DG-categories.
\end{exams}

To better understand the notion of $(\lambda,\ce)$-purity in this context, we need to introduce some notation. 

\begin{defi}
	Given a $\cv$-category $\ck$ and an object $P\in\cv$, we call a map $$f\colon P\to\ck(X,Y)$$ a {\em $P$-morphism} from $X$ to $Y$, and denote it by $(f,P)\colon X\to Y$. 
\end{defi}

When $P=I$ we recover the standard notion of morphism $f\colon X\to Y$ in the underlying category $\ck_0$.

\begin{exam}\label{DG}
	In the monoidal category $\bo{DGAb}$ of chain complexes, consider the chain complex $P_n$ which has $\mathbb{Z}$ in degree $n$ and $n-1$, with differential $d_n=\tx{id}$, and which is trivial in every other degree. In~\cite{NST2020cauchy} these $P_n$ are called $S^nL\mathbb{Z}$. It is easy to see that to give a morphism $P_n\to A$ in $\bo{DGAb}$ is the same as specifying an element $x\in A_n$ in degree $n$.\\
	Then, according to the nomenclature in \cite{NST2020cauchy}, a $P_n$-morphism in a DG-category $\ck$ is a protomorphism of degree $n$. It is easy to see that the set of all $P_n$'s is a regular generator of $\bo{DGAb}$ made of finitely presentable and regular projective objects (in fact, this is a dense generator by~\cite[3.10]{NST2020cauchy}).
\end{exam}

\begin{rem}
	The morphisms above were considered in \cite{W} to express every $\cv$-category as a category enriched over $[\cv^{\tx{op}},\bo{SET}]$, for an opportune universe enlargement $\bo{SET}$ of $\bo{Set}$. In particular, the definition below can be seen as part of the composition rule in $[\cv^{\tx{op}},\bo{SET}]$-categories.
\end{rem}

Morphisms and $P$-morphisms can be composed as follows:

\begin{defi}\label{PI-comp}
	Given a $P$-morphism $(u,P)\colon X\to Y$ and a morphism $g\colon Y\to Z$ in a $\cv$-category $\ck$, the composite of $(u,P)$ and $g$ is the $P$-morphism
	$$ (gu,P)\colon P\stackrel{\cong}{\longrightarrow}I\otimes P\stackrel{g\otimes u}{\xrightarrow{\hspace*{0.8cm}}}\ck(Y,Z)\otimes \ck(X,Y)\stackrel{\circ}{\longrightarrow}\ck(X,Z). $$
	Similarly, given a morphism $g\colon X\to Y$ and a $P$-morphism $(u,P)\colon Y\to Z$, the composite of $g$ and $(u,P)$ is the $P$-morphism
	$$ (ug,P)\colon P\stackrel{\cong}{\longrightarrow}P\otimes I\stackrel{u\otimes g}{\xrightarrow{\hspace*{0.8cm}}}\ck(Y,Z)\otimes \ck(X,Y)\stackrel{\circ}{\longrightarrow}\ck(X,Z). $$
	
\end{defi}

\begin{rem}
	Equivalently, in the first case, $(gu,P)$ can be defined as the composite
	$$ P\stackrel{u}{\longrightarrow}\ck(X,Y)\stackrel{\ck(X,g)}{\xrightarrow{\hspace*{0.8cm}}}\ck(X,Z). $$
	Similarly for $(ug,P)$. However, the approach given above makes clear the relationship with Remark~\ref{QI}.
\end{rem}

Thanks to this we can talk about commutative squares involving $P$-morphisms: 

\begin{defi}
	We say that the square below commutes
	\begin{center}
		\begin{tikzpicture}[baseline=(current  bounding  box.south), scale=2]
			
			\node (a0) at (0,-0.8) {$Y$};
			\node (b0) at (1,-0.8) {$Y'$};
			\node (c0) at (0,0) {$X$};
			\node (d0) at (1,0) {$X'$};
			
			\path[font=\scriptsize]
			
			(a0) edge [->] node [below] {$f$} (b0)
			(a0) edge [<-] node [left] {$(u,P)$} (c0)
			(b0) edge [<-] node [right] {$(v,P)$} (d0)
			(c0) edge [->] node [above] {$g$} (d0);
		\end{tikzpicture}	
	\end{center} 
	if $(fu,P)=(vg,P)$. 
\end{defi}

Note that, to give a commutative square as above is the same as giving a map $P\to\ck^\to(g,f)$; then $(u,P)$ and $(v,P)$ can be recovered by taking the two projections out of $\ck^\to(g,f)$.

We are now ready to give a characterization of $\ce$-purity in terms of $P$-morphisms.

\begin{propo}\label{E-pure}
	A morphism $f\colon K\to L$ in a $\cv$-category $\ck$ is $(\lambda,\ce)$-pure if and only if for any commutative square 
	\begin{center}
		\begin{tikzpicture}[baseline=(current  bounding  box.south), scale=2]
			
			\node (a0) at (0,-0.8) {$K$};
			\node (b0) at (1,-0.8) {$L$};
			\node (c0) at (0,0) {$A$};
			\node (d0) at (1,0) {$B$};
			
			\path[font=\scriptsize]
			
			(a0) edge [->] node [below] {$f$} (b0)
			(a0) edge [<-] node [left] {$(u,P)$} (c0)
			(b0) edge [<-] node [right] {$(v,P)$} (d0)
			(c0) edge [->] node [above] {$g$} (d0);
		\end{tikzpicture}	
	\end{center} 
	with $A,B\in\ck_\lambda$ and $P\in \cp$, there exists $(t,P)\colon B\to K$ such that $(tg,P)=(u,P)$.
\end{propo}
\begin{proof}
	A morphism $f\colon K\to L$ is barely $(\lambda,\ce)$-pure if and only if for any $g\colon A\to B$ in $\ck_\lambda$, the map $q\colon \ck(B,K)\to\cq(g,f)$  of Notation~\ref{notation} is a regular epimorphism. Since $\cp$ is a strong generator of regular projective objects, this holds if and only if $\cv_0(P,q)$ is a surjection of sets for any $P\in\cp$. Now, $\cq(g,f)$ is defined as the $(\ce,\cm)$ factorization of the first projection $$\pi_1\colon\ck^\to(g,f)\to\ck(A,K);$$ thus $\cv_0(P,\cq(g,f))$ is the standard image factorization of the function $\cv_0(P,\pi_1)$. It follows that $\cv_0(P,\cq(g,f))$ can be described as the set of those $P$-morphisms $(u,P)\colon A\to K$ which can be completed to a commutative square as depicted in the statement. 
	
	In conclusion, the map $\cv_0(P,q)$ is surjective if and only if for any $(u,P)$ as above there exists $(t,P)\colon B\to K$ such that $(tg,P)=(u,P)$.
\end{proof}

\begin{propo}\label{E-pure-powers}
	Let $\ck$ be a $\cv$-category. The following are equivalent for a morphism $f\colon K\to L$ in $\ck$:\begin{enumerate}
		\item $f$ is $(\lambda,\ce)$-pure.
		\item Assuming powers by $\cp$ exist in $\ck$, the morphism $P\pitchfork f$ is $\lambda$-pure in the ordinary sense for any $P\in\cp$.
		\item Assuming copowers by $\cp$ exist in $\ck$, for any commutative square 
		\begin{center}
			\begin{tikzpicture}[baseline=(current  bounding  box.south), scale=2]
				
				\node (a0) at (0,-0.8) {$K$};
				\node (b0) at (1.1,-0.8) {$L$};
				\node (c0) at (0,0) {$P\cdot A$};
				\node (d0) at (1.1,0) {$P\cdot B$};
				
				\path[font=\scriptsize]
				
				(a0) edge [->] node [below] {$f$} (b0)
				(a0) edge [<-] node [left] {$u'$} (c0)
				(b0) edge [<-] node [right] {$v'$} (d0)
				(c0) edge [->] node [above] {$P\cdot g$} (d0);
			\end{tikzpicture}	
		\end{center} 
		with $g\colon A\to B$ in $\ck_\lambda$ and $P\in \cp$, there exists $t\colon P\cdot B\to K$ such that $t\circ(P\cdot g)=u'$.
	\end{enumerate}
\end{propo}
\begin{proof}
	Since maps $P\cdot A\to B$ in $\ck$ correspond to morphisms $P\to\ck(A,K)$ in $\cv$, to give a square as in $(3)$ is the same as giving a commutative square involving $P$-morphisms as in Proposition~\ref{E-pure} above. Therefore $(1)\Leftrightarrow (3)$ follows immediately.
	
	For $(1)\Leftrightarrow (2)$ it is again enough to notice that, acting by transposition, a square as in Proposition~\ref{E-pure} is the same as a commutative square
	\begin{center}
		\begin{tikzpicture}[baseline=(current  bounding  box.south), scale=2]
			
			\node (a0) at (0,-0.8) {$P\pitchfork K$};
			\node (b0) at (1.1,-0.8) {$P\pitchfork L$};
			\node (c0) at (0,0) {$A$};
			\node (d0) at (1.1,0) {$B$};
			
			\path[font=\scriptsize]
			
			(a0) edge [->] node [below] {$P\pitchfork f$} (b0)
			(a0) edge [<-] node [left] {$u$} (c0)
			(b0) edge [<-] node [right] {$v$} (d0)
			(c0) edge [->] node [above] {$g$} (d0);
		\end{tikzpicture}	
	\end{center} 
	and a lifting for the square in Proposition~\ref{E-pure} is the same as a lifting for the square above. Thus $(1)$ is equivalent to $P\pitchfork f$ being $\lambda$-pure in the ordinary sense for every $P\in\cp$.
\end{proof}

\begin{rem}
	It follows that our $(\omega,\ce)$-pure morphisms are the same as the $\cp$-pure morphisms of \cite[Definition~6.4]{LT20}.
\end{rem}

When $\ck$ is locally $\lambda$-presentable the first part of the statement below follows from Corollary~\ref{pure-Epure}.

\begin{propo}\label{pure-E-pure}
	Let $\ck$ be a $\cv$-category with copowers by $\cp$. Then every ordinary $\lambda$-pure morphism $f\colon K\to L$ in $\ck$ is $(\lambda,\ce)$-pure. If the unit $I$ of $\cv$ is regular projective, then every $(\lambda,\ce)$-pure morphism is ordinarily $\lambda$-pure.
\end{propo}
\begin{proof}
	If $f\colon K\to L$ is ordinarily $\lambda$-pure then it satisfies point $(3)$ of Proposition~\ref{E-pure-powers} since copowers by finitely presentable objects of $\cv$ are still finitely presentable. Therefore $f$ is $(\lambda,\ce)$-pure.  
	
	If the unit $I$ is regular projective, then we can assume it to be an element of $\cp$ (we can always take $\cp$ to be the set of all the finitely presentable and regular projective objects). Thus, $f\cong I\pitchfork f$ is ordinarily $\lambda$-pure by Proposition~\ref{E-pure-powers}.
\end{proof}

\begin{exam}
	When $\cv=\bo{Ab}$ is the category of abelian groups, the unit $\mathbb Z$ is regular projective; thus the enriched notion of $\ce$-purity coincides with the ordinary one, which was already studied by Prest in the context of definable additive category and model theory \cite{Pre11}.
\end{exam}

Recall the notion of $\ce$-injectivity from Section~\ref{enr-purity1}. In this context it translates into the following: given a $\cv$-category $\ck$, an object $X$ is {\em $\ce$-injective} with respect to $h\colon A\to B$ in $\ck$ if 
	$$\ck(h,X)\colon\ck(B,X)\to\ck(A,X)$$ 
is a regular epimorphism. In other words, if for any $(f,P)\colon A\to X$, with $P\in\cp$, there exists $(g,P)\colon B\to X$ such that $(gh,P)=(f,P)$.

We can now prove the following theorem characterizing enriched $\ce$-injectivity classes. This will generalize \cite[Proposition~6.8]{LT20} characterizing such classes only for those $\cv$ with a regular projective unit. 

\begin{theo}\label{char}
	Let $\cv$ be a symmetric monoidal $\lambda$-quasivariety endowed with the (regular epi, mono) factorization system. For a locally $\lambda$-presentable $\cv$-category $\ck$, the $(\lambda,\ce)$-injectivity classes of $\ck$ are precisely the full subcategories of $\ck$ closed under products, powers by $\cp$, $\lambda$-filtered colimits, and $(\lambda,\ce)$-pure subobjects.
\end{theo}
\begin{proof}
	Note that Assumption~\ref{assumption} is satisfied, and that a morphism is $(\lambda,\ce)$-pure if and only if it is barely $(\lambda,\ce)$-pure Lemma~\ref{pure-bpure} since regular epimorphisms are stable under pullbacks in $\cv$. Moreover, a map $e$ in $\cv$ is a regular epimorphism if and only if $\cv_0(P,e)$ is surjective for any $P\in\cp$ (since $\cp$ forms a regular generator), and powers by $\cp$ are $\ce$-stable by \cite[Remark~4.15]{LT20}. 
	
	Thus we can apply Theorem~\ref{char-theo-1} for $\cg:=\cp$ and obtain the desired characterization of $(\lambda,\ce)$-injectivity classes.
\end{proof}

\begin{rem}
	Note that, in presence of $\lambda$-filtered colimits and products, the existence of powers by $\cp$ implies that of powers by any regular projective objects of $\cv$. Indeed, if $X\in\cv$ is projective, it is a retract of coproducts of elements $P_i$ of $\cp$ \cite[Proposition~4.8]{LT20}; thus the power $X\pitchfork K$ of an object $K$ in a $\cv$-category $\ck$ will be a retract of the product of the powers $P_i\pitchfork K$.
\end{rem}

\begin{rem}
	Every $\ce$-injectivity class in $\ck$ is an ordinary injectivity class in the underlying category $\ck_0$. However, the converse needs not hold in general since ordinary injectivity classes may not have (enriched) absolute colimits, while they exist in every $\ce$-injectivity class (since regular epimorphisms are closed under absolute colimits in $\cv^\to$).\\
	Consider $\cv=\bo{DGAb}$, and let $C_k$ be the chain complex having $\mathbb{Z}/k\mathbb{Z}$ in degree $0$ and $-1$, with differential $d_0=\tx{id}$, and which is trivial in every other degree. To give a map $C_k\to A$ is the same as giving an element $x\in A_0$ such that $kx=0$. Let $h\colon P_0\to C_k$ be the component-wise projection (where $P_0$ was defined in Example~\ref{DG}); then the ordinary injectivity class defined by $h$ in $\bo{DGAb}$ is the full subcategory of all chain complexes $A$ such that $kx=0$ for every $x\in A_0$. This is not an $\ce$-injectivity class in $\bo{DGAb}$ since it is not closed under suspensions (shifts of the degrees), which can be expressed as absolute colimits (see \cite{NST2020cauchy}).
\end{rem}

In the proof of result below, at some point we will need to see a morphism $x\colon X\to Y$ as a $P$-morphism $(x',P)\colon X\to Y$ for some $P\in\cp$. When the unit is regular projective this is trivial (since we can assume $I\in\cp$), but in general we need to argue differently.

\begin{nota}\label{I*}
	Since $\cv$ is a $\lambda$-quasivariety we can write the unit $I$ as a regular quotient of a coproduct of $\lambda$-presentable regular projective objects. Moreover, since $I$ is $\lambda$-presentable, we can assume such coproduct to be $\lambda$-small. But $\lambda$-small coproducts of $\lambda$-presentable regular projective objects is still $\lambda$-presentable regular projective. It follows that there exists a $\lambda$-presentable and regular projective object $I^*$ together with a regular epimorphism $$e\colon I^*\twoheadrightarrow I.$$ We shall now fix such $I^*$ and $e$ and, without loss of generality, assume that $I^*\in\cp$ (one can always enlarge $\cp$ to contain $I^*$).
\end{nota}

\begin{rem}\label{QI}
	Given $I^*$ as above, by pre-composing with $e$, every morphism in a $\cv$-category $\ck$ can be seen as a $I^*$-morphism. 
	
	As it happens in the case of ordinary morphisms (Definition~\ref{PI-comp}), $I^*$-morphisms can be composed with $P$-morphisms for any $P\in\cp$. Indeed, given $P\in\cp$, tensoring with $e$ gives a regular epimorphism $$e_P\colon P\otimes I^*\to P$$ which, since $P$ is regular projective, splits giving $s_e\colon P\to P\otimes I^*$ such that $e_Ps_P=1_P$. Now, given a $I^*$-morphism $(g,I^*)\colon X\to Y$ and a $P$-morphism $(f,P)\colon Y\to X$, we define their composite as the $P$-morphism
	$$ (fg,P)\colon P\stackrel{s_e}{\longrightarrow}P\otimes I^*\stackrel{f\otimes g}{\xrightarrow{\hspace*{0.8cm}}}\ck(Y,Z)\otimes \ck(X,Y)\stackrel{\circ}{\longrightarrow}\ck(X,Z). $$
	Finally, note that the composite of a a morphism $g\colon X\to Y$ with a $P$-morphism $(f,P)\colon Y\to Z$ with is the same as the composite of the $I^*$-morphism $(ge,I^*)$, induced by $g$ through $e$, with $(f,P)$; this is because $s_P$ is a section of $e_P$.
\end{rem}

\begin{exam}
	When $\cv=\bo{DGAb}$ we can take $I^*=P_0$ (as in Example~\ref{DG}) together with its projection down to $I$. This allows us to see every morphism in a DG-category as a protomorphism of degree $0$.
\end{exam}

Recall the notion of $\ce$-split monomorphism from Definition~\ref{e-split}. Below we denote by $(1_X,I^*)\colon X\to X$ the $I^*$-morphism obtained from the identity of $X$ in the sense of Remark~\ref{QI}.

\begin{propo}\label{E-spli-char}
	A morphism $s\colon X\to Y$ in a $\cv$-category $\ck$ is an $\ce$-split monomorphism if and only if there exists $(t,I^*)\colon Y\to X$ for which $$(ts,I^*)=(1_X,I^*)$$
	in the sense of Definition~\ref{PI-comp}.
\end{propo}
\begin{proof}
	Assume first that $s\colon X\to Y$ is an $\ce$-split monomorphism, so that $\ck(s,X)$ is a regular epimorphism. Then, since $I^*$ is regular projective, the function $$\cv_0(I^*,\ck(s,X))\colon \cv_0(I^*,\ck(Y,X))\to\cv_0(I^*,\ck(X,X))$$
	is surjective. By taking an element that is mapped to $(1_X,I^*)$ we find $(t,I^*)\colon Y\to X$ as in the statement.
	
	Conversely, suppose that there exist $(t,I^*)\colon Y\to X$ for which $(ts,I^*)=(1_X,I^*)$; in the setting of Remark~\ref{QI} the equality can be rewritten as
	$$(t(se),I^*)=(1_X,I^*)$$
	where $(se,I^*)$ is the $I^*$-morphism induced by precomposing $s\colon I\to \ck(X,Y)$ with $e\colon I^*\to I$. To prove that $s$ is an $\ce$-split monomorphism it is enough to show that 
	$$\cv_0(P,\ck(s,X))\colon \cv_0(P,\ck(Y,X))\to\cv_0(P,\ck(X,X))$$
	is surjective for any $P\in\cp$. Given $(f,P)\colon X\to X$ in the codomain, we can consider the $P$-morphism $(ft,P)\colon Y\to X$ defined in Remark~\ref{QI} above. By definition, $\cv_0(P,\ck(s,X))$ sends $(ft,P)$ to the the composite
	\begin{align*}
		(ft,P)\circ (s,I)&= (ft,P)\circ (se,I^*)\\
		&=(f,P)\circ(t(se),I^*)\\ 
		&= (f,P)\circ(1_X,I^*)\\
		&=(f,P),
	\end{align*}
	where the first composition is in the sense of Definition~\ref{PI-comp} and the others are in the sense of Remark~\ref{QI}. It follows that $\cv_0(P,\ck(s,X))$is surjective, and thus $\ck(s,X)$ is a regular epimorphism and $s$ an $\ce$-split monomorphism.
\end{proof}

\begin{exam}
	When $\cv=\bo{DGAb}$ an $\ce$-split monomorphism is just a protosplit monomorphism in the sense of \cite{NST2020cauchy}.
\end{exam}

\begin{propo}\label{filt colimit of E-split}
	If $\ck$ is a locally $\lambda$-presentable $\cv$-category, then:\begin{enumerate}
		\item $\ce$-split morphisms are $(\lambda,\ce)$-pure;
		\item $\lambda$-filtered colimits of $(\lambda,\ce)$-pure morphisms in $\ck^\to$ are $(\lambda,\ce)$-pure;
		\item every $(\lambda,\ce)$-pure morphism in $\ck$ is a $\lambda$-filtered colimit of $\ce$-split monomorphisms in $\ck^\to$.
	\end{enumerate} 
\end{propo}
\begin{proof}
	(1) follows from Corollary~\ref{E-split-pure} and (2) is a consequence of Proposition~\ref{filt-col-pure} since monomorphisms are closed under $\lambda$-filtered colimits in $\cv^\to$.
	
	(3) We apply the same proof of \cite[Proposition~2.30]{AR} with some changes where the notion of $\ce$-split morphism is needed. 
	
	Given a $(\lambda,\ce)$-pure morphism $f\colon A\to B$, we can write it as a $\lambda$-filtered colimit in $\ck^\to$ of maps $f_i\colon A_i\to B_i$ between $\lambda$-presentable objects with connecting morphisms $(u_i,v_i)\colon f_i\to f$. As usual, all the morphisms $u_i$ and $v_i$ can be seen as $I^*$-morphisms (by pre-composing with $e\colon I^*\to I$), giving commutative squares as below.
	\begin{center}
		\begin{tikzpicture}[baseline=(current  bounding  box.south), scale=2]
			
			\node (a0) at (0,-0.8) {$A$};
			\node (b0) at (1,-0.8) {$B$};
			\node (c0) at (0,0) {$A_i$};
			\node (d0) at (1,0) {$B_i$};
			
			\path[font=\scriptsize]
			
			(a0) edge [->] node [below] {$f$} (b0)
			(a0) edge [<-] node [left] {$(u_i,I^*)$} (c0)
			(b0) edge [<-] node [right] {$(v_i,I^*)$} (d0)
			(c0) edge [->] node [above] {$f_i$} (d0);
		\end{tikzpicture}	
	\end{center}
	Since $f$ is $(\lambda,\ce)$-pure, we obtain $I^*$-morphisms $(t_i,I^*)\colon B_i\to A$ for which $(t_if,I^*)=(u_i,I^*)$. Consider now the pushouts below.
	\begin{center}
		\begin{tikzpicture}[baseline=(current  bounding  box.south), scale=2]
			
			\node (a0) at (0,-0.8) {$A$};
			\node (b0) at (1,-0.8) {$\bar{B}_i$};
			\node (c0) at (0,0) {$A_i$};
			\node (d0) at (1,0) {$B_i$};
			
			\path[font=\scriptsize]
			
			(a0) edge [->] node [below] {$\bar f_i$} (b0)
			(a0) edge [<-] node [left] {$u_i$} (c0)
			(b0) edge [<-] node [right] {$\bar u_i$} (d0)
			(c0) edge [->] node [above] {$f_i$} (d0);
		\end{tikzpicture}	
	\end{center}
	It is shown in \cite[Proposition~2.30]{AR} that the colimit of the $\bar f_i$ is still $f$; thus we only need to show that each $\bar f_i$ is $\ce$-split.
	
	We can see the identity on $A$ as an $I^*$-morphism and hence, acting by transposition, as a map $\bar 1_A\colon A\to I^*\pitchfork A$; similarly each $(t_i,I^*)$ corresponds to a morphism $\bar t_i\colon B_i\to I^*\pitchfork A$. The fact that $(t_if,I^*)=(u_i,I^*)$ means that the square below commutes.
	\begin{center}
		\begin{tikzpicture}[baseline=(current  bounding  box.south), scale=2]
			
			\node (a0) at (0,-0.8) {$A$};
			\node (b0) at (1,-0.8) {$I^*\pitchfork A$};
			\node (c0) at (0,0) {$A_i$};
			\node (d0) at (1,0) {$B_i$};
			
			\path[font=\scriptsize]
			
			(a0) edge [->] node [below] {$\bar 1_A$} (b0)
			(a0) edge [<-] node [left] {$u_i$} (c0)
			(b0) edge [<-] node [right] {$\bar t_i$} (d0)
			(c0) edge [->] node [above] {$f_i$} (d0);
		\end{tikzpicture}	
	\end{center}
	By the universal property of the pushout we obtain a morphism $\bar g_i\colon\bar B_i\to I^*\pitchfork A$ such that (in particular) $\tilde 1_A= \bar g_i \bar f_i$. Acting by transposition, $\bar g_i$ corresponds to an $I^*$-morphism $(g_i,I^*)\colon \bar B_i\to A$ such that $(1_A,I^*)=(g_i\bar f_i,I^*)$. Thus each $\bar f_i$ is an $\ce$-split monomorphism by Proposition~\ref{E-spli-char}.
\end{proof}

\appendix

\section{The canonical language}\label{can-lang}

Fix a factorization system $(\ce,\cm)$ on $\cv$. In this section we consider a notion of language on the underlying category of a given $\cv$-category as in \cite[Chapter~18]{Pre11}.

Given such a language $\mathbb L$, an $\mathbb L$-structure $M$ is the data of an object $M_S$ in $\cv$ for any sort $S$ in $\mathbb{L}$, and a morphism $M_f\colon M_S\to M_T$ for any function symbol $f\colon S\to T$ in $\mathbb L$.

\begin{defi}
	Let $\ck$ be a locally $\lambda$-presentable $\cv$-category. The canonical language $\mathbb L(\ck_\lambda)$ of $\ck$ is the language with sorts $s_A$ the objects $A$ of $\ck_\lambda$ and function symbols $s_f\colon s_A\to s_B$ corresponding to morphisms $f\colon B\to A$ in $\ck_\lambda$. 
\end{defi}

\begin{rem}
	Each object $K$ of $\ck$ defines an $\mathbb L(\ck_\lambda)$-structure as follows: every sort $s_A\in \mathbb L(\ck_\lambda)$ is interpreted as the object $\ck(A,K)$, and every function symbol $s_f\in\mathbb L(\ck_\lambda)$ as above is interpreted as the morphism $\ck(f,K)\colon \ck(A,K)\to \ck(B,K)$.
\end{rem}

Given a language $\mathbb L$, the atomic formulas are those of the form 
$$ \phi(x,y)\equiv (f(x)=g(y))$$
with $f\colon S\to U$, $g\colon T\to U$, and $x$ and $y$ of sort $S$ and $T$ respectively.
A primitive positive formula (pp-formula) is one of the form 
$$ \psi(x)\equiv \exists y\ \phi(x,y) $$
where each $\phi$ is a conjunction of atomic formulas.

For any $\mathbb L$-structure $M$, we interpret atomic formulas $ \phi(x,y)\equiv (f(x)=g(y))$ as the $\cm$-subobject of $M_S\times M_T$
given by the $(\ce,\cm)$-factorization of the map
\begin{center}
	\begin{tikzpicture}[baseline=(current  bounding  box.south), scale=2]
		
		\node (a0) at (-0.1,0.8) {$M_{f,g}$};
		\node (c0) at (2,0.8) {$ M_S\times M_T $};
		\node (d0) at (1,0.2) {$\parallel \phi\parallel_M$};
		
		\path[font=\scriptsize]
		
		(a0) edge [->] node [above] {$h$} (c0)
		(a0) edge [->>] node [below] {$\ce\ \ \ $} (d0)
		(d0) edge [>->] node [below] {$\ \ \ \cm$} (c0);
	\end{tikzpicture}	
\end{center}
where $M_{f,g}$ is the pullback of $M_f$ along $M_g$ and $h$ is the morphism induced into the product.

\begin{rem}\label{proper-int}
	If the factorization system is proper then $\cm$ contains all the regular monomorphisms (see the dual of \cite[2.1.4]{FK})). Thus $\parallel \phi\parallel_M\cong M_{f,g}$.
\end{rem}

The interpretation of a conjunction $\phi_1\wedge\cdots\wedge\phi_n$ of atomic formulas is given by the $\cm$-intersection of each $\parallel\phi_i\parallel_M$. Finally, the interpretation of $ \psi(x)\equiv\exists y\ \phi(x,y) $ is given by the $(\ce,\cm)$-factorization of the composite below.
\begin{center}
	\begin{tikzpicture}[baseline=(current  bounding  box.south), scale=2]
		
		\node (a0) at (-0.1,0.8) {$\parallel \phi\parallel_M$};
		\node (b0) at (1.3,0.8) {$ M_S\times M_T $};
		\node (c0) at (2.6,0.8) {$M_S$};
		\node (d0) at (1.3,0.2) {$\parallel \psi\parallel_M$};
		
		\path[font=\scriptsize]
		
		(a0) edge [>->] node [above] {} (b0)
		(b0) edge [->] node [above] {$\pi_1$} (c0)
		(a0) edge [->>] node [below] {$\ce\ \ \ $} (d0)
		(d0) edge [>->] node [below] {$\ \ \ \cm$} (c0);
	\end{tikzpicture}	
\end{center}

\begin{rem}\label{lfp-inter}
	Consider an object $K$ in a locally $\lambda$-presentable $\ck$ with its standard $\mathbb L(\ck_\lambda)$-structure. The interpretation of the atomic formula $ \phi(x,y)\equiv (s_f(x)=s_g(y))$, for $f\colon C\to A$ and $g\colon C\to B$ in $\ck_\lambda$, is given by the $(\ce,\cm)$-factorization of the map
	\begin{center}
		\begin{tikzpicture}[baseline=(current  bounding  box.south), scale=2]
			
			\node (a0) at (-0.1,0.8) {$\widehat\ck(f,g,K)$};
			\node (c0) at (2.2,0.8) {$ \ck(A,K)\times\ck(B,K) $};
			\node (d0) at (1.1,0.2) {$\parallel \phi\parallel_K$};
			
			\path[font=\scriptsize]
			
			(a0) edge [->] node [above] {$h$} (c0)
			(a0) edge [->>] node [below] {$\ce\ \ \ $} (d0)
			(d0) edge [>->] node [below] {$\ \ \ \cm$} (c0);
		\end{tikzpicture}	
	\end{center}
	where $\widehat\ck(f,g,K)$ is the pullback of $\ck(f,K)$ along $\ck(g,K)$ and $h$ is the morphism induced into the product.
	
	Since $\ce$ is closed under composition, the interpretation of a formula of the form $\psi(x)\equiv \exists y\ (s_f(x)=s_g(y))$ is given by the $(\ce,\cm)$-factorization of the map
	\begin{center}
		\begin{tikzpicture}[baseline=(current  bounding  box.south), scale=2]
			
			\node (a0) at (-0.1,0.8) {$\widehat\ck(f,g,K)$};
			\node (c0) at (2.2,0.8) {$ \ck(A,K)$};
			\node (d0) at (1.1,0.2) {$\parallel \psi\parallel_K$};
			
			\path[font=\scriptsize]
			
			(a0) edge [->] node [above] {$\widehat g$} (c0)
			(a0) edge [->>] node [below] {$\ce\ \ \ $} (d0)
			(d0) edge [>->] node [below] {$\ \ \ \cm$} (c0);
		\end{tikzpicture}	
	\end{center}
	where $\widehat g$ is the map opposite to $\ck(g,K)$ in the pullback defining $\widehat\ck(f,g,K)$. When $f=1_A$, then $\parallel \psi\parallel_K$ is the $(\ce,\cm)$-factorization of the map $\ck(g,K)$.
\end{rem}

\begin{defi}
	Let $f\colon K\to L$ be a morphism in a locally $\lambda$-presentable $\cv$-category $\ck$. We say that $f$ is {\em elementary with respect to a pp-formula $\psi$} in the language $\mathbb L(\ck_\lambda)$ if the induced diagram
	\begin{center}
		\begin{tikzpicture}[baseline=(current  bounding  box.south), scale=2]
			
			\node (a0) at (0,0.8) {$\parallel \psi\parallel_K$};
			\node (b0) at (1.3,0.8) {$\ck(A,K)$};
			\node (c0) at (0,0) {$\parallel \psi\parallel_L$};
			\node (d0) at (1.3,0) {$\ck(A,L)$};
			
			\path[font=\scriptsize]
			
			(a0) edge [>->] node [above] {} (b0)
			(a0) edge [->] node [left] {} (c0)
			(b0) edge [->] node [right] {$\ck(A,f)$} (d0)
			(c0) edge [>->] node [below] {} (d0);
		\end{tikzpicture}	
	\end{center}
	is a pullback.
\end{defi}

\begin{propo}\label{elem}
	Let $\ck$ be a locally $\lambda$-presentable $\cv$-category and $f\colon K\to L$ be a morphism in it. Then: \begin{enumerate}
		\item The map $f$ is $(\lambda,\ce)$-pure if and only if it is elementary with respect to any pp-formula of the form
		$$\psi(x)\equiv\exists y\ (s_{h}(x)=s_{g}(y))$$
		in $\mathbb L(\ck_\lambda)$.
		\item If $(\ce,\cm)$ is proper, then $f$ is $(\lambda,\ce)$-pure if and only if it is elementary with respect to any pp-formula in $\mathbb L(\ck_\lambda)$.
	\end{enumerate}
\end{propo}
\begin{proof}
	(1)  First, consider $g\colon A\to B$ in $\ck_\lambda$ and the pp-formula $$\psi(x)\equiv\exists y\ (s_{1_A}(x)=s_{g}(y)).$$ 
	We obtain a commutative diagram
	\begin{center}
		\begin{tikzpicture}[baseline=(current  bounding  box.south), scale=2]

			\node (a) at (-1.3,0.8) {$\ck(B,K)$};
			\node (b) at (-1.3,0) {$\ck(B,L)$};
			
			\node (a0) at (0,0.8) {$\parallel \psi\parallel_K$};
			\node (b0) at (1.3,0.8) {$\ck(A,K)$};
			\node (c0) at (0,0) {$\parallel \psi\parallel_L$};
			\node (d0) at (1.3,0) {$\ck(A,L)$};
			
			\path[font=\scriptsize]
			
			(a) edge [->>] node [above] {} (a0)
			(a) edge [->] node [left] {$\ck(B,f)$} (b)
			(b) edge [->>] node [right] {} (c0)
			
			(a0) edge [>->] node [above] {} (b0)
			(a0) edge [->] node [left] {} (c0)
			(b0) edge [->] node [right] {$\ck(A,f)$} (d0)
			(c0) edge [>->] node [below] {} (d0);
		\end{tikzpicture}	
	\end{center}
	where $\parallel \psi\parallel_K$ and $\parallel \psi\parallel_L$ are by definition the $(\ce,\cm)$-factorizations of $\ck(g,K)$ and $\ck(g,L)$ respectively, and the middle vertical arrow is induced by the factorization. Then by definition we have $\parallel \psi\parallel_L\cong \cp(g,L)$ as $\cm$-subobjects of $\ck(A,L)$.
	
	Now, $f$ is $(\lambda,\ce)$-pure with respect to $g$ if and only if the $(\ce,\cm)$-factorization of $\ck(g,K)$ is given by $\cp(g,f)$, if and only if $\parallel \psi\parallel_K\cong\cp(g,f)$ as $\cm$-subobjects of $\ck(A,K)$, if and only if $\parallel \psi\parallel_K$ is the pullback of $\parallel \psi\parallel_L$ along $\ck(A,f)$, if and only if $f$ is elementary with respect to $\psi$.

	To conclude it is enough to show that every formula of the form
	$$\phi(x,y)\equiv (s_{h}(x)=s_{g}(y))$$
	is equivalent in $\ck$ to one
	$$\phi'(x,y)\equiv (s_{1_A}(x)=s_{h'}(y))$$
	involving an identity morphism, meaning that $\parallel\phi\parallel_X\cong \parallel\phi'\parallel_X$. Given $\phi$ as above, with $g\colon C\to A$ and $h\colon C\to B$, let $D$ be the pushout of $g$ along $h$ in $\ck$, together with the induced morphism $h'\colon A\to D$. Then $D$ is still $\lambda$-presentable and $\phi$ is equivalent in $\ck$ to the formula $\phi'$ with the $h'$ just defined. This is thanks to Remark~\ref{lfp-inter} and the fact that $\widehat\ck(g,h,X)\cong \ck(D,X)$ for any $X\in\ck$. 
	
	(2) Let $(\ce,\cm)$ be proper, and consider now a general pp-formula 
$$
\psi(x)\equiv \exists y\ (\phi_1(x,y)\wedge\cdots\wedge\phi_n(x,y))
$$
	where $\phi_i(x,y)\equiv (s_{g_i}(x)=s_{h_i}(y))$, with $g_i\colon C_i\to A$ and $h_i\colon C_i\to B$ in $\ck_\lambda$.\\ 
	Consider the $\lambda$-presentable object $C:=\sum_{i\leq n}C_i$ and the morphisms $g\colon C\to A$ and $h\colon C\to B$ induced by the $g_i$'s and the $h_i$'s respectively. We wish to prove that $\phi_1\wedge\cdots\wedge\phi_n$ is equivalent in $\ck$ to
	$$ \phi(x,y)\equiv (s_{g}(x)=s_{h}(y)); $$
	meaning that $\parallel\phi_1\wedge\cdots\wedge\phi_n\parallel_X\cong \parallel\phi\parallel_X$ as $\cm$-subobjects of $\ck(A,X)\times\ck(B,X)$, for any $X\in\ck$. Note that, by Remark~\ref{proper-int}, $\parallel\phi_i\parallel_X$ is the pullback of $\ck(g_i,X)$ along $\ck(h_i,X)$; thus $\parallel\phi_1\wedge\cdots\wedge\phi_n\parallel_X$, being obtained by pulling back all of the $\parallel\phi_i\parallel_X$, coincides with the pullback of $\ck(g,X)$ along $\ck(h,X)$, which is just $\parallel\phi\parallel_X$. Now the result follows from point (1).
\end{proof}

In the context of barely $(\lambda,\ce)$-pure morphisms we can prove the following:

\begin{propo}
	Let $(\ce,\cm)$ be proper and $\ck$ be a locally $\lambda$-presentable $\cv$-category. \begin{enumerate}
		\item Any morphism that is elementary with respect to any pp-formula of the form
		$$\psi(x)\equiv\exists y\ (s_{h}(x)=s_{g}(y))$$
		in the language $\mathbb L(\ck_\lambda)$, is barely $(\lambda,\ce)$-pure.
		\item If every barely $(\lambda,\ce)$-pure morphism is $(\lambda,\ce)$-pure, then a morphism is barely $(\lambda,\ce)$-pure if and only it is elementary with respect to any pp-formula in the language $\mathbb L(\ck_\lambda)$.
	\end{enumerate}
\end{propo}
\begin{proof}
			
			
			
			
			
			
	%
	Follows from the previous proposition and Lemma~\ref{pure-bpure}.
\end{proof}

\end{document}